\documentclass{article}
\usepackage{amsmath}
\usepackage{amsfonts}
\usepackage{amssymb}
\usepackage{amsthm}
\usepackage{mathtools}
\usepackage{bm} 
\usepackage{enumitem}
\usepackage{graphicx}
\usepackage{epstopdf}
\usepackage{geometry}
\usepackage{hyperref}
\usepackage{inputenc}
\usepackage{xcolor}
\usepackage{tikz}

\title{Controlling Canard Cycles}
\author{Hildeberto Jard\'on-Kojakhmetov and
Christian Kuehn
}

\newtheorem{definition}{Definition}
\newtheorem{remark}{Remark}
\newtheorem{theorem}{Theorem}
\newtheorem{proposition}{Proposition}
\newtheorem{lemma}{Lemma}

\newcommand{\ve}{\varepsilon}
\newcommand{\R}{\mathbb{R}}
\newcommand{\N}{\mathbb{N}}
\newcommand{\bbS}{\mathbb{S}}

\newcommand{\cA}{\mathcal{A}}
\newcommand{\cB}{\mathcal{B}}
\newcommand{\cC}{\mathcal{C}}

\newcommand{\cE}{\mathcal{E}}

\newcommand{\cI}{\mathcal{I}}

\newcommand{\cM}{\mathcal{M}}
\newcommand{\cN}{\mathcal{N}}
\newcommand{\cO}{\mathcal{O}}

\newcommand{\cS}{\mathcal{S}}

\newcommand{\cW}{\mathcal{W}}

\newcommand{\txten}{\textnormal{en}}
\newcommand{\txtex}{\textnormal{ex}}
\newcommand{\txta}{\textnormal{a}}
\newcommand{\txtr}{\textnormal{r}}
\newcommand{\txtD}{\textnormal{D}}
\newcommand{\cl}{\textnormal{cl}}
\newcommand{\txtmin}{\textnormal{min}}
\newcommand{\txtmax}{\textnormal{max}}

\newcommand{\bx}{\bar{x}}
\newcommand{\by}{\bar{y}}
\newcommand{\bz}{\bar{z}}
\newcommand{\br}{\bar{r}}
\newcommand{\bu}{\bar{\mu}}
\newcommand{\bv}{\bar{v}}
\newcommand{\be}{\bar{\ve}}
\newcommand{\ba}{\bar{\alpha}}

\newcommand{\ch}[1]{{\textcolor{black}{#1}}}
\newcommand{\chh}[1]{{\textcolor{black}{#1}}}

\begin{document}

\maketitle

\begin{abstract}
Canard cycles are periodic orbits that appear as special solutions of fast-slow systems (or singularly perturbed ordinary differential equations). It is well known that canard cycles are difficult to detect, hard to reproduce numerically, and that they are sensible to exponentially small changes in parameters. In this paper we combine techniques from geometric singular perturbation theory, the blow-up method, and control theory, to design controllers that stabilize canard cycles of planar fast-slow systems with a folded critical manifold. As an application, we propose a controller that produces stable mixed-mode oscillations in the van der Pol oscillator. 
\end{abstract}

\section{Introduction}\label{sec:introduction}

Fast-slow systems (also known as singularly perturbed ordinary differential equations, see more details in Section 
\ref{sec:preliminaries}) are often used to model phenomena occurring in two or more time scales. Examples of these 
are vast and range from oscillatory patters in biochemistry and neuroscience 
\cite{ermentrout2010mathematical,izhikevich2007dynamical,banasiak2014methods,hek2010geometric}, 
all the way to stability analysis and control of power networks \cite{chow1982time,dorfler2012synchronization}, 
among many others \cite[Chapter 20]{kuehn2015multiple}. The overall idea behind the analysis of fast-slow systems 
is to separate the behavior that occurs at each time scale, understand such behavior, and then try to elucidate 
the corresponding dynamics of the full system. Many approaches have been developed, such as asymptotic methods 
\cite{eckhaus2011matched,kevorkian2012multiple,OMalley1991,sanders2007averaging}, numeric and computational tools 
\cite{hairer2010solving,kaper2015geometry}, and geometric techniques 
\cite{fenichel1979geometric,jones1995geometric,kaper1999intro}, see also 
\cite{kuehn2015multiple,mishchenko1980differential,verhulst2005methods}. In this article we take a geometric approach.

Although the time scale separation approach has been very fruitful, there are some cases in which it does not suffice to completely describe the dynamics of a fast-slow system, see the details in Section \ref{sec:preliminaries}. The reason is that, for some systems, the fast and the slow dynamics are interrelated in such a way that some complex behavior is only discovered when they are not fully separated. An example of the aforementioned situation are the so-called \emph{canards} \cite{benoit1981chasse,benoit1990canards,dumortier1996canard}, see Section \ref{sec:canards} for the appropriate definition. Canards are orbits that, counter-intuitively, stay close for a considerable amount of time to a repelling set of equilibrium points of the fast dynamics. Canards are extremely important not only in the theory of fast-slow systems, but also in applied sciences, and especially in neuroscience, as they have allowed, for example, the detailed description of the very fast onset of large amplitude oscillations due to small changes of a parameter in neuronal models \cite{ermentrout2010mathematical,izhikevich2007dynamical} and of other complex oscillatory patterns \cite{brons1991canard,desroches2012mixed,moehlis2002canards}. Due to their very nature, canard orbits are not robust, meaning that small perturbations may drastically change the shape of the orbit.

On the other hand the application of singular perturbation techniques in control theory is far-reaching. Perhaps, as already introduced above, one of the biggest appeals of the theory of fast-slow systems is the time scale separation, which allows the reduction of large systems into lower dimensional ones for which the control design is simpler \cite{jardon2017model,kokotovic1976singular,kokotovic1980singular}. Applications range from the control of robots \cite{siciliano1988singular,spong1990control,jardon2016model}, all the way to industrial biochemical processes, and large power networks \cite{dmitriev2006singular,kokotovic1999singular,kokotovic1984applications,kurina2017discrete,naidu2001singular,naidu2002singular}. However, as already mentioned, not all fast-slow systems can be analyzed by the convenient time scale separation strategy, and although some efforts from very diverse perspectives have been made \cite{ARTSTEIN20171603,1531-3492_2019_8_3653,artstein1997tracking,artstein2002singularly,Fridman2001ADS,Fridman2001StatefeedbackHC,jardon2017model,jardon2019stabilization}, a general theory that includes not only the regulation problem but also the path following and trajectory planning problems is, to date, lacking.

The main goal of this article is to merge techniques of fast-slow dynamical systems with control theory methods to develop controllers that stabilize canard orbits. The idea of controlling canards has already been explored in \cite{durham2008feedback}, where an integral feedback controller is designed for the FitzHugh-Nagumo model to steer it towards the so-called ``canard regime''. In contrast, here we take a more general and geometric approach by considering the \ch{folded} canard normal form, see Section \ref{sec:canards}. Moreover, we integrate control techniques with Geometric Singular Perturbation Theory (GSPT) and propose a controller design methodology in the blow-up space. Later we apply such geometric insight to the van der Pol oscillator where we provide a controller that produces any oscillatory pattern allowed by the geometric properties of the model, see Section \ref{sec:vdp}.

The rest of this document is arranged as follows: in Section \ref{sec:preliminaries} we present definitions and preliminaries of the geometric theory of fast-slow systems and of \ch{folded} canards, which are necessary for the main analysis. In Section \ref{sec:fold} we develop a controller that stabilizes folded canard orbits, where the main strategy is to combine the blow-up method with state-feedback control techniques to achieve the goal. Afterwards in Section \ref{sec:vdp}, as an extension to our previously developed controller, we develop a controller that stabilizes several canard cycles and is able to produce robust complex oscillatory patters in the van der Pol oscillator. We finish in Section \ref{sec:conclusions} with some concluding remarks and an outlook. 


\section{Preliminaries}\label{sec:preliminaries}
A fast-slow system is a singularly perturbed ordinary differential equation (ODE) of the form
\begin{equation}\label{eq:sfs1}
	\begin{split}
		\ve\dot x &= f(x,y,\ve,\lambda)\\
		\dot y    &= g(x,y,\ve,\lambda),
	\end{split}
\end{equation}
where $x\in\R^m$ is the \emph{fast variable}, $y\in\R^n$ the \emph{slow variable}, $0<\ve\ll1$ is a small parameter accounting for the time scale separation between the aforementioned variables, $\lambda\in\R^p$ denotes other parameters, and $f$ and $g$ are assumed sufficiently smooth. In this document the over-dot is used to denote the derivative with respect to the \emph{slow} time $\tau$. It is well-known that, for $\ve>0$, an equivalent way of writing \eqref{eq:sfs1} is
\begin{equation}\label{eq:sfs2}
	\begin{split}
		x' &= f(x,y,\ve,\lambda)\\
		y' &= \ve g(x,y,\ve,\lambda),
	\end{split}
\end{equation}
where now the prime denotes the derivative with respect to the \emph{fast} time $t\coloneqq \tau/\ve$. 

One of the mathematical theories that is concerned with the analysis of \eqref{eq:sfs1}-\eqref{eq:sfs2} is Geometric Singular Perturbation Theory (GSPT) \cite{kuehn2015multiple}. The overall idea of GSPT is to study the limit equations that result from setting $\ve=0$ in \eqref{eq:sfs1}-\eqref{eq:sfs2}. Then, one looks for invariant objects that can be shown to persist up to small perturbations. Such invariant objects give a qualitative description of the behavior of \eqref{eq:sfs1}-\eqref{eq:sfs2}. Accordingly, setting $\ve=0$ in \eqref{eq:sfs1}-\eqref{eq:sfs2} one gets
\begin{equation}\label{eq:reduced}
	\begin{split}
		0 &= f(x,y,0,\lambda)\\
		\dot y    &= g(x,y,0,\lambda)
	\end{split}\qquad\hspace{2cm}\qquad
	\begin{split}
		x' &= f(x,y,0,\lambda)\\
		y' &= 0,
	\end{split}
\end{equation}
known, respectively, as the \emph{reduced slow subsystem} (which is a Constrained Differential Equation \cite{takens1976constrained} or a Differential Algebraic Equation \cite{kunkel2006differential}) and \emph{the layer equation}. The aforementioned limit systems are not equivalent any more, but they are related by the following important geometric object.
\begin{definition}[The critical manifold] The critical manifold is defined as
\begin{equation}
	\cC_0=\left\{ (x,y)\in\R^m\times\R^n\,|\, f(x,y,0,\lambda)=0 \right\}.
\end{equation}
\end{definition}

We note that the critical manifold is the phase-space of the reduced slow subsystem and the set of equilibrium points of the layer equation. The properties of the critical manifold are essential to GSPT, in particular the following.
\begin{definition}[Normal hyperbolicity]
	Let $p\in\cC_0$. We say that $p$ is hyperbolic if the matrix $\txtD_xf(p,0,\lambda)|_{\cC_0}$ has all its eigenvalues away from the imaginary axis. If every point $p\in\cC_0$ is hyperbolic, we say that $\cC_0$ is \emph{normally hyperbolic}. On the contrary, if for some $p\in\cC_0$ the matrix $\txtD_xf(p,0,\lambda)|_{\cC_0}$ has at least one of its eigenvalues on the imaginary axis, then we say that $p$ is a non-hyperbolic point.
\end{definition}

It is known from Fenichel's theory \cite{fenichel1971persistence,fenichel1979geometric} that a \emph{compact and normally hyperbolic} critical manifold $\cS_0\subseteq\cC_0$ of \eqref{eq:reduced} persists as a locally invariant \emph{slow} manifold $\cS_\ve$ under sufficiently small perturbations. In other words, Fenichel's theory guarantees that in a neighborhood of a normally hyperbolic critical manifold the dynamics of \eqref{eq:sfs1}-\eqref{eq:sfs2} are well approximated by the limit systems \eqref{eq:reduced}. 
\begin{remark}
	Along this paper we use the notation $\cS_0^\txta$ and $\cS_0^\txtr$ to denote, depending on the eigenvalues of $\txtD_xf(x,y,0,\lambda)|_{\cS_0}$, the attracting an repelling parts of the (compact) critical manifold $\cS_0$. Accordingly, the corresponding slow manifolds are denoted as $\cS_\ve^\txta$ and $\cS_\ve^\txtr$.
\end{remark}

On the other hand, critical manifolds may lose normal hyperbolicity, for example, due to singularities of the layer equation, see Figure \ref{fig:fold-red-flow}. It is in fact due to loss of normal hyperbolicity that, as in this paper, some interesting and complicated dynamics may arise in seemingly simple fast-slow systems. Fenichel's theory, however, does not hold in the vicinity of non-hyperbolic points. In some cases, depending on the nature of the non-hyperbolicity, the \emph{blow-up method}~\cite{jardon2019survey} is a suitable technique to analyze the complicated dynamics that arise. In the forthcoming section we introduce the particular type of orbits that we are concerned with and that arise due to loss of normal-hyperbolicity of the critical manifold; the so-called \emph{canards}. 

\subsection{\ch{Planar Folded Canards}}\label{sec:canards} 

\ch{In this section we briefly describe folded canards and folded canard cycles in the plane. As we mention below, the adjective ``folded'' is due to a fold singularity. However, we remark that canards (and canard cycles) can be related to other types of singularities}. The interested reader is refereed to, e.g. \cite{dumortier1996canard,krupa2001extending,wechselberger2012propos}, references therein and, in particular, \cite[Chapter 8]{kuehn2015multiple} and \cite[Section 3]{jardon2019survey} for more detailed information.

Let us start by recalling that the canonical form of a canard point \cite{krupa2001extending} is given by
\begin{equation}\label{eq:fold1}
	\begin{split}
		x' &= -y h_1(x,y,\ve,\alpha) + x^2h_2(x,y,\ve,\alpha) + \ve h_3(x,y,\ve,\alpha)\\
		y' &= \ve\left(xh_4(x,y,\ve,\alpha)-\alpha h_5(x,y,\ve,\alpha) +y h_6(x,y,\ve,\alpha)\right),
	\end{split}
\end{equation}
where $(x,y)\in\R^2$, $0<\ve\ll1$, and \ch{$\alpha$ is a parameter}. Furthermore
\begin{equation}\label{eq:fold-hot}
	\begin{split}
		h_3(x,y,\ve,\alpha) &= \cO(x,y,\ve,\alpha)\\
		h_i(x,y,\ve,\alpha) &= 1+\cO(x,y,\ve,\alpha),\qquad i=1,2,4,5,
	\end{split}
\end{equation}
and $h_6$ is smooth. For simplicity of notation, we rewrite \eqref{eq:fold1} together with \eqref{eq:fold-hot} as
\begin{equation}\label{eq:fold2}
	\begin{split}
		x' &= -y+x^2 + \widetilde f(x,y,\ve,\alpha)\\
		y' &= \ve(x-\alpha + \widetilde g(x,y,\ve,\alpha)),
	\end{split}
\end{equation}
where \chh{$\widetilde f$ and $\widetilde g$ denote the corresponding higher order terms, that is:
\begin{equation}\label{eq:fg}
	\begin{split}
		\widetilde f(x,y,\ve,\alpha) &\coloneqq  -y\cO(x,y,\ve,\alpha)+x^2\cO(x,y,\ve,\alpha)+\ve\cO(x,y,\ve,\alpha),\\
		\widetilde g(x,y,\ve,\alpha) &\coloneqq  x\cO(x,y,\ve,\alpha)-\alpha\cO(x,y,\ve,\alpha) +yh_6(x,y,\ve,\alpha).
	\end{split}
\end{equation}
}
The critical manifold is locally (near the origin) a \ch{perturbed} parabola and is given by
\begin{equation}
	\cS_0 = \left\{ (x,y)\in\R^2\,|\, -y+x^2+ \widetilde f(x,y,0,\alpha)=0\right\}.
\end{equation}
The (slow and fast) reduced flow corresponding to \eqref{eq:fold2} is as shown in Figure \ref{fig:fold-red-flow}.
\begin{figure}[htbp]
	\begin{tikzpicture}
		\node at (0,0){
		\includegraphics[scale=.65]{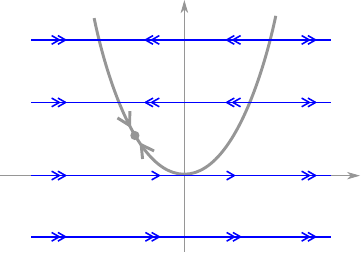}
		};
		\node at (1.9,-.35) {$x$};
		\node at (0.07,1.55) {$y$};
		\node at (-1.,1.5) {$\cS_0^\txta$};
		\node at ( 1,1.5) {$\cS_0^\txtr$};
		\node[anchor=center] at (0.07,-2){$\alpha<0$};
	\end{tikzpicture}\hfill
	\begin{tikzpicture}
		\node at (0,0){
		\includegraphics[scale=.65]{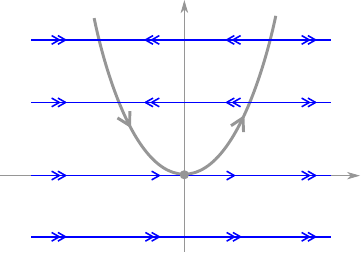}
		};
		\node at (1.9,-.35) {$x$};
		\node at (0.07,1.55) {$y$};
		\node at (-1.,1.5) {$\cS_0^\txta$};
		\node at ( 1,1.5) {$\cS_0^\txtr$};
		\node[anchor=center] at (0.07,-2){$\alpha=0$};
	\end{tikzpicture}\hfill
	\begin{tikzpicture}
		\node at (0,0){
		\includegraphics[scale=.65]{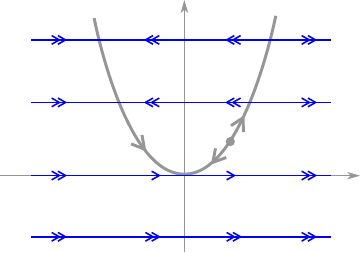}
		};
		\node at (1.9,-.35) {$x$};
		\node at (0.07,1.55) {$y$};
		\node at (-1.,1.5) {$\cS_0^\txta$};
		\node at ( 1,1.5) {$\cS_0^\txtr$};
		\node[anchor=center] at (0.07,-2){$\alpha>0$};
	\end{tikzpicture}
	\caption{Singular flow of \eqref{eq:fold2} near the origin. The gray parabola depicts the critical manifold $\cS_0$ which is partitioned in its attracting $\cS_0^\txta=\cS_0|_{\left\{ x<0\right\}}$ and repelling $\cS_0^\txtr=\cS_0|_{\left\{ x>0\right\}}$ parts, while the origin (the fold point) is non-hyperbolic. If $\alpha=0$ the origin is also called \emph{canard point}. In this latter case, the orbit along the critical manifold is also known as \emph{singular maximal canard}.}
	\label{fig:fold-red-flow}
\end{figure}
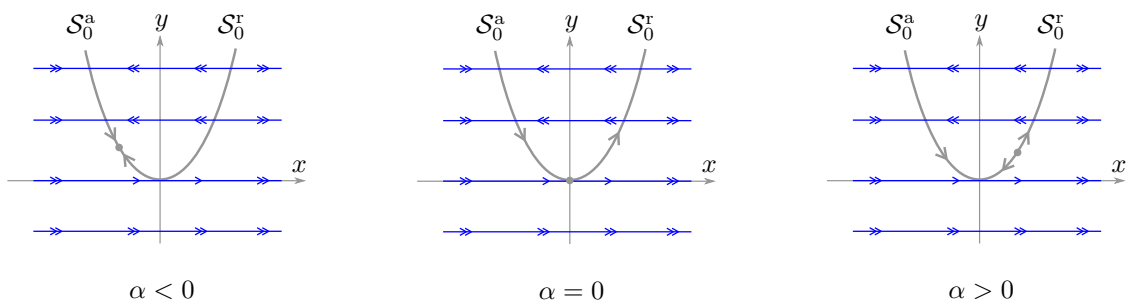

\begin{remark}\label{rem:H}
	To fix ideas, consider for a moment \eqref{eq:fold2} with zero higher order terms\footnote{Refer to \cite{krupa2001extending} for the much more complicated case that includes the higher order terms.}, that is
\begin{equation}\label{eq:fold3}
	\begin{split}
		x' &= -y+x^2\\
		y' &= \ve(x-\alpha).
	\end{split}
\end{equation}
Then, it is straightforward to check that, for $\ve>0$ and $\alpha=0$, the orbits of \eqref{eq:fold3} are given by level sets of 
\begin{equation}\label{eq:H}
	H(x,y,\ve)=\frac{1}{2}\exp\left( -\frac{2y}{\ve}\right)\left( \frac{y}{\ve}-\frac{x^2}{\ve}+\frac{1}{2} \right).
\end{equation}

Some orbits of \eqref{eq:fold3} are shown in Figure \ref{fig:canardsH}\ch{, and in fact it is known \cite{krupa2001extending} that canard cycles exist for $H\in(0,\frac{1}{4})$.}

\begin{figure}[htbp]\centering
\hspace*{1.5cm} 
	\begin{tikzpicture}
		\node at (0,0){
			\includegraphics{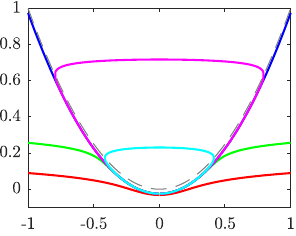}
		};
		\node[blue,anchor=west] at    (2.5,1.7){\small $H=0$};
		\node[magenta,anchor=west] at (2.5,1.){\small $H=1\times10^{-10}$};
		\node[cyan,anchor=west] at (2.5,.25){\small $H=.01$};
		\node[green,anchor=west] at (2.5,-.45){\small $H=-.01$};
		\node[red,anchor=west] at (2.5,-.95){\small $H=-10$};
		\node at (0.2,-2.2){$x$};
		\node at (-2.9,0.2){$y$};
		\node at (-1.6,1.5){$\cS_0^\txta$};
		\node at ( 2,1.5){$\cS_0^\txtr$};
	\end{tikzpicture}
	\caption{Orbits of \eqref{eq:fold3} obtained as level sets of \eqref{eq:H}. The dashed gray curve is the critical manifold. Compare with $\alpha=0$ in Figure \ref{fig:fold-red-flow}.}
	\label{fig:canardsH}
\end{figure}
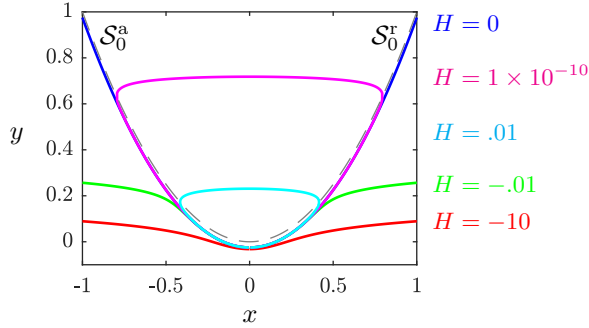
\end{remark}

What is remarkable is that there are orbits that closely follow the unstable branch of the critical manifold for slow time of order $\cO(1)$. Such type of orbits are known as \emph{canards}. There is a particular canard, which is called \emph{maximal canard} and is given by $\left\{ H=0\right\}$ that connects the attracting slow manifold $\cS_\ve^\txta$ with the repelling one $\cS_\ve^\txtr$. More relevant to this paper are periodic orbits with canard portions, which called \emph{canard cycles}.

In the following section we design feedback controllers for \eqref{eq:fold1} that render a particular canard cycle asymptotically stable. In other words, we consider the path following control problem where a canard orbit is the reference.


\section{Controlling Folded Canards}\label{sec:fold}
We propose to study two control problems, namely
\begin{equation}\label{eq:fast_c_example}
	\begin{split}
		x' &= -y+x^2+ \widetilde f(x,y,\ve,\alpha) +u(x,y,\ve,\alpha)\\
		y' &= \ve(x-\alpha+\widetilde g(x,y,\ve,\alpha)),
	\end{split}
\end{equation}
which we call \emph{the fast control problem} and
\begin{equation}
	\begin{split}
		x' &= -y+x^2+ \widetilde f(x,y,\ve,\alpha) \\
		y' &= \ve(x-\alpha+\widetilde g(x,y,\ve,\alpha)+u(x,y,\ve,\alpha)),
	\end{split}
\end{equation}
to be referred to as \emph{the slow control problem}. Recall that $\widetilde f$ and $\widetilde g$ stand for the higher order terms as in \eqref{eq:fg}. The objective is to stabilize a certain reference canard cycle to be denoted by $\gamma_h$.

\clearpage

\begin{remark}\leavevmode
\begin{itemize}[leftmargin=*]
	\item The choice of the above control problems is motivated by applications, especially in neuron models, see \cite{durham2008feedback,izhikevich2007dynamical,ermentrout2010mathematical}, where the input current appears in the fast (voltage) variable and regulates the distinct firing patterns. However, if one is interested in the fully actuated case, a combination of the techniques presented here shall also be useful.
	\item Throughout this document we assume that one has full knowledge of the functions $\widetilde f$ and $\widetilde g$. This means that for the fast (resp. slow) control problem we assume $\widetilde f=0$ (resp. $\widetilde g=0$). Otherwise one considers a controller of the form $u=-\widetilde f+v$ (resp. $u=-\widetilde g+v$) where now $v$ is to be designed.
\end{itemize}
\end{remark}

 \ch{Notice that in the case of the fast-control problem \eqref{eq:fast_c_example}, the controller changes the fast dynamics. This means that the controller can change the type of singularities the critical manifold may present.
	To be more precise, consider for a moment \eqref{eq:fast_c_example} with $u=-kx$, $k>0$, a simple proportional feedback controller. The closed-loop system then reads as
	\begin{equation}
		\begin{split}
		x' &= -kx-y+x^2+ \widetilde f(x,y,\ve,\alpha)\\
		y' &= \ve(x-\alpha+\widetilde g(x,y,\ve,\alpha)),
		\end{split}
	\end{equation}
	for which the origin is now normally hyperbolic. This means that the feedback controller has changed the type of singularity (at the origin) from a fold to a regular one. It is clear that these type of controllers are not compatible with our task. So, we shall design controllers that \emph{do not change the type of singularity} of the open-loop system. To formalize what we mean by ``not changing the type of singularity'', let us first recall the following definition:
	\begin{definition}[$k$-jet equivalence]
		Let $F:\R^n\to\R^n$ and $G:\R^n\to\R^n$ be smooth maps. We say that $F$ and $G$ are ($k$-jet) equivalent at $p\in\R^n$ if $F(p)=G(p)$ and $F(x)-G(x)=\mathcal{O}(||x-p||^{k+1})$ as $x\to p$. An equivalence class defined by the previous notion of equivalence is called the $k$-jet of $F$ at $p$, and shall be denoted by $j^kF(p)$ \cite{arnold2012singularities}. 	
	\end{definition}
	Next we have a formal definition of what we refer to as a \emph{compatible controller}:
\begin{definition}[Compatible controller]\label{def:comp}
	Consider a control system
	\begin{equation}
		\dot\zeta = f(\zeta,\lambda,u), 
	\end{equation}
	where $\zeta\in\R^n$ is the state variable, $\lambda\in\R^p$ denotes system parameters (possibly including $0<\ve\ll1$) and $u\in\R^m$ stands for the controller. Suppose that for the open-loop system, that is when $u=0$, the origin $\zeta=0\in\R^n$ is a nilpotent equilibrium point of $\dot\zeta = f(\zeta,0,0)$ and that there is a $k\in\N$ such that $k$ is the smallest number so that $j^kf(0)\neq0$. Let $u$ be a state-feedback controller, that is \chh{$u=u(\zeta,\lambda,\ell)$, where $\ell\in\R^m$ stands for parameters of the controller such as controller gains,} and denote by $\dot\zeta=F(\zeta,\lambda,\ell)$ the closed-loop system. We say that $u$ is a \emph{compatible controller} if the open-loop vector field $f(\zeta,\lambda,0)$ and the closed-loop vector field $F(\zeta,\lambda,\ell)$ are \emph{$k$-jet equivalent} at the origin for $\lambda=0$.
\end{definition}
We emphasize that once one fixes coordinates on $\R^n$, a $k$-jet equivalence between two maps means that such maps coincide on their partial derivatives up to order $k$.\\
As an example of the above definition, recall that a planar fast-slow system with a generic fold at the origin is given by
\begin{equation}
	\begin{split}
		x' &= f(x,y,\ve)\\
		y' &= \ve g(x,y,\ve), 
	\end{split}
\end{equation}
with the defining conditions $f(0,0,0)=0$, $\frac{\partial f}{\partial x}(0,0,0)=0$, $\frac{\partial^2 f}{\partial x^2}(0,0,0)\neq0$, and the non-degeneracy condition $\frac{\partial f}{\partial y}(0,0,0)\neq0$. Next, let $u=u(x,y,\ve)$ be a state-feedback controller and suppose one considers the fast-slow control system 
\begin{equation}
	\begin{split}
		x' &= \underbrace{f(x,y,\ve)+u(x,y,\ve)}_{=:F(x,y,\ve)}\\
		y' &= \ve g(x,y,\ve).
	\end{split}
\end{equation}
Then, $u$ is a compatible controller if the closed-loop system verifies: $F(0,0,0)=0$, $\frac{\partial F}{\partial x}(0,0,0)=0$, $\frac{\partial^2 F}{\partial x^2}(0,0,0)\neq0$, and $\frac{\partial F}{\partial y}(0,0,0)\neq0$, which implies that the controller does not change the class of the singularity, since the origin is still a fold point of the closed-loop system.
}


\subsection{The fast control problem}\label{sec:fast}
\newcommand{\tx}{\hat{x}}
\newcommand{\tu}{\hat{u}}
\newcommand{\tg}{\hat{g}}
In this section we study the control problem defined by
\begin{equation}\label{eq:fast-a}
	\begin{split}
		x' &= -y + x^2 + u(x,y,\ve,\alpha)\\
		y' &= \ve(x-\alpha + \widetilde g(x,y,\ve,\alpha)).
	\end{split}
\end{equation} 

Due to the fact that the slow dynamics are not actuated, we are going to stabilize canards centered at $(x,y)=(\alpha,0)$. Then, it is convenient to define $\tx=x-\alpha$, which brings \eqref{eq:fast-a} into
\begin{equation}\label{eq:fast-b}
	\begin{split}
		\tx' &= -y+(\tx+\alpha)^2+\tu(\tx,y,\ve,\alpha)\\
		y'   &= \ve(\tx+\tg(\tx,y,\ve,\alpha)),
	\end{split}
\end{equation}
where $\tu(\tx,y,\ve,\alpha)=u(\tx+\alpha,y,\ve,\alpha)$ and similarly for $\tg$\chh{, we have $\tg(\tx,y,\ve,\alpha)=g(\tx+\alpha,y,\ve,\alpha)$}.

\begin{theorem}\label{thm:fold-fast}
	Consider \eqref{eq:fast-b} and let $\hat H=H(\tx,y,\ve)$ be defined by \eqref{eq:H}. Then, the following hold:
	\begin{enumerate}[leftmargin=*]

		\item \chh{The compatible controller 
		\begin{equation}\label{eq:thm_u1}
			\hat u=-2\alpha\tx-\alpha^2+c_1\tx\ve^{1/2}\exp(c_2y\ve^{-1})(\hat H-h),
		\end{equation}
		where $c_1>0$, $c_2\in\R$ and $h\leq\frac{1}{4}$, renders the canard orbit $\hat{\gamma}_h=\left\{ (\tx,y)\in\R^2\, | \, \hat H=h \right\}$ locally asymptotically stable for $\ve>0$ sufficiently small.
		}


		\item Let $\hat\Gamma\subset\R^2$ be a neigborhood of $\hat\gamma_h$ for $h\in(0,\frac{1}{4})$. Suppose that, additionally to \eqref{eq:fg}, $\hat g$ is of the form $\hat g=\tx\hat\phi(\tx,y,\ve,\alpha)$ for some  function $\hat\phi$, and that \chh{$\hat\phi\neq1$} for all $(\hat x,y)\in\hat\Gamma$. Then \chh{the compatible controller 
		\begin{equation}\label{eq:thm_u2}
			\hat u=-2\alpha\tx-\alpha^2+c_1\tx\ve^{1/2}(\hat H-h)\exp(c_2y\ve^{-1})-(y-\tx^2)\hat\phi,
		\end{equation}
		where $c_1>0$, $c_2\in\R$} renders the canard orbit $\hat{\gamma}_h=\left\{ (\tx,y)\in\R^2\, | \, \hat H=h \right\}$ locally \chh{(within $\hat\Gamma$)} asymptotically stable.
	\end{enumerate}
\end{theorem}
\ch{
	\begin{remark}\leavevmode
	\begin{itemize}
		\item The choice of the controller gain $c_2$ in Theorem \ref{thm:fold-fast} has an important impact in numerical simulations due to the fact of it appearing as an argument of the exponential function. The choice $c_2=2$ yields the better numerical results when stabilizing canard cycles, that is for $h\in(0,\frac{1}{4})$. However, to stabilize the maximal canard ($h=0$), it is necessary to choose $c_2<2$ to ensure that the controller remains bounded as $y\to\infty$. See more detail in section \ref{sec:K1}.
		\item We recall that although from Theorem \ref{thm:fold-fast} one is able to stabilize any canard (because $h\leq\frac{1}{4}$), canard cycles exist only for $h\in(0,\frac{1}{4})$, see Figure \ref{fig:canardsH} and \cite{krupa2001extending}.
		\item The second item of Theorem \ref{thm:fold-fast} holds for any $\ve>0$.
	\end{itemize}		
	\end{remark}
}

The proof of Theorem \ref{thm:fold-fast} follows from the forthcoming analysis and is summarized in section \ref{proof:fold-fast}. We show in Figure \ref{fig:thm-fast} a simulation of the results contained in Theorem \ref{thm:fold-fast}.

\begin{figure}[htbp]\centering
	\begin{tikzpicture}
		\node at (0,0){
			\begin{tikzpicture}
				\node at (0,0){
					\includegraphics{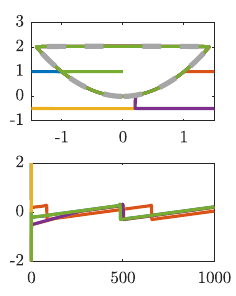}
				};
				\node at (0,-3){(a)};
			\end{tikzpicture}
		};
		\node at (5,0){
			\begin{tikzpicture}
				\node at (0,0){
					\includegraphics{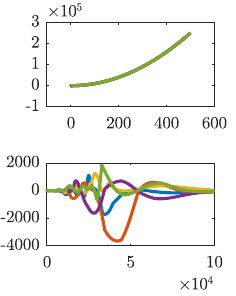}
				};
				\node at (0,-3){(b)};
			\end{tikzpicture} 
		};
		\node at (10,0){
			\begin{tikzpicture}
				\node at (0,0){
					\includegraphics{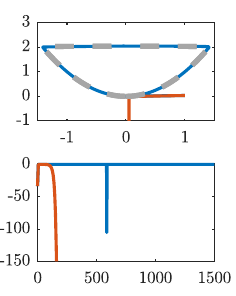}
				};
				\node at (0,-3){(c)};
			\end{tikzpicture} 
		};
	\end{tikzpicture}
	\caption{In all three columns we show, in the first row the $(\hat x,y)$ phase portrait of the closed-loop system \eqref{eq:fast-b} and in the second row the time-series of the corresponding controller. In all these simulations $\ve=0.01$. (a) The case for which $\hat g=0$ and with parameters $(\alpha,c_1,c_2,h)=(-0.1,1,2,\frac{1}{4}e^{-400})$. We remark here that in order for the constant $h=\frac{1}{4}e^{-400}$ to be numerically feasible one has to input $h\exp(c_2y\ve^{-1})=\frac{1}{4}\exp(c_2y\ve^{-1}-400)$ into the numerical integration algorithm. The desired canard cycle to be followed is shown in dashed-gray. (b) The maximal canard case with $\hat g=0$ and with parameters $(\alpha,c_1,c_2,h)=(0,1,2-e^{-15},0)$. Notice that, indeed, trajectories follow the unstable branch $\cS_\ve^\txtr$ for a large ``height'' and that the corresponding controller remains bounded. (c) An example of the effect of the extra term in \eqref{eq:thm_u2} where we show two trajectories with the same initial conditions. The unstable one is obtained with the controller \eqref{eq:thm_u1} while the stable one with \eqref{eq:thm_u2}. The desired canard cycle to be followed is shown in dashed-gray. The large spike in the controller is observed every time the trajectory crosses the $y$-axis long a fast fiber. For such simulation we have used $(\alpha,c_1,c_2,h)=(0,5,2,\frac{1}{4}e^{-400})$ and $g=100x(y-x^2)$. For more details see Sections \ref{sec:foldK2} and \ref{sec:K1}.}
	\label{fig:thm-fast}
\end{figure}
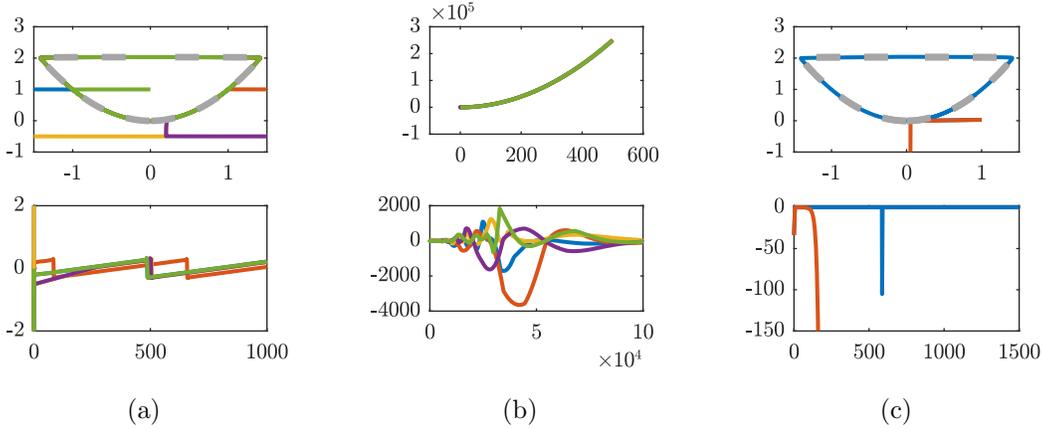

As already anticipated, the idea is to design the controller $\tu$ in the blow-up space.  Therefore, let us consider a coordinate transformation defined by
	\begin{equation}\label{eq:bu-fast}
		\tx=\br\bx,\quad y=\br^2\by, \quad \ve=\br^2\be, \quad \tu=\br^2\bu, \quad \alpha=\br\ba,
	\end{equation}
	where $(\bx,\by,\be,\bu,\ba)\in\mathbb{S}^{4}$\ch{, with $\bbS^4$ denoting the $4$-sphere, that is $\{\bx^2+\by^2+\be^2+\bu^2+\ba^2=1\}$,} and $\br\in[0,\infty)$. \ch{As is usual with the blow-up method \cite{jardon2019survey}, instead of working in spherical coordinates, we consider local coordinates in local charts. In our particular context, these local charts parametrize different hemi-spheres of $\bbS^4$.} Analogous to the analysis of the canard point in \cite{krupa2001extending} we consider the charts \ch{$K_1=\left\{ \by=1 \right\}$ and $K_2=\left\{ \be=1 \right\}$. To distinguish the local coordinates in these charts, we let $(r_1,x_1,\ve_1,\mu_1,\alpha_1)$ be local coordinates in $K_1$, and $(r_2,x_2,y_2,\mu_2,\alpha_2)$ be local coordinates in $K_2$. In this way, these local coordinates are defined by:}
	\begin{equation}\label{eq:Ks}
		\begin{aligned}
			K_1:& \quad \tx=r_1x_1, \quad y=r_1^2, \quad\quad \ve=r_1^2\ve_1, \quad \tu=r_1^2\mu_1, \quad \alpha=r_1\alpha_1, \qquad\qquad \textnormal{\ch{for $y\geq0$}},\\
			K_2:& \quad \tx=r_2x_2, \quad y=r_2^2y_2, \quad \ve=r_2^2, \quad \quad \tu=r_2^2\mu_2, \quad \alpha=r_2\alpha_2, \qquad\qquad \textnormal{\ch{for $\ve\geq0$}}.
		\end{aligned}
	\end{equation}

\ch{In particular, it is worth noting that in chart $K_1$ the coordinate $r_1$ is a rescaling of the ``original coordinate'' $y$ for $y\geq0$, while in chart $K_2$, the coordinate $r_2$ is a rescaling of $\ve\geq0$. Furtheremore, in a qualitative sense, in chart $K_1$ one studies trajectories of \eqref{eq:fast-b} as they approach and leave a small neighborhood of the fold point in the positive $y$ direction, while in chart $K_2$ one investigates the trajectories of \eqref{eq:fast-b} within a sufficiently small neighborhood of the fold point.}

The coordinates in the above charts are related by the \emph{transition maps}:
\begin{equation}\label{eq:match12}
		\kappa_{12}:K_1\to K_2, \qquad r_2=r_1\ve_1^{1/2}, \; x_2=x_1\ve_1^{-1/2}, \; y_2=\ve_1^{-1}, \; \mu_2=\mu_1\ve_1^{-1}, \; \alpha_2=\alpha_1\ve_1^{-1/2},
\end{equation}
for $\ve_1>0$ and
\begin{equation}\label{eq:match21}
	\kappa_{21}:K_2\to K_1,\qquad r_1=r_2y_2^{1/2}, \; x_1=x_2y_2^{-1/2},\; \mu_1=\mu_2y_2^{-1}, \; \ve_1=y_2^{-1}, \; \alpha_1=\alpha_2y_2^{-1/2},
\end{equation}
for $y_2>0$.

\subsubsection{Analysis in the rescaling chart \texorpdfstring{$K_2$}{K2}}\label{sec:foldK2}

\renewcommand{\bx}{x_2}
\renewcommand{\by}{y_2}
\renewcommand{\bz}{z_2}
\renewcommand{\br}{r_2}
\renewcommand{\bu}{\mu_2}
\renewcommand{\bv}{\nu_2}
\renewcommand{\be}{\ve_2}
\renewcommand{\ba}{\alpha_2}
\newcommand{\bg}{g_2}
\newcommand{\bphi}{\phi_2}
\newcommand{\bpsi}{\psi_2}

The blown-up (and desingularized) local vector field in this chart reads as
\begin{equation}\label{eq:K2-1}
	\begin{split}
		\bx' &= -\by + (\bx+\ba)^2 + \bu \\
		\by' &= \bx + \bg,
	\end{split}
\end{equation}
\chh{where $\bg=\bg(\br,\bx,\by,\ba)$ is smooth and defined by the blow-up of $\hat g$. More precisely, from \eqref{eq:fg} and keeping in mind the usual desingularization step, one has  that
\begin{equation}\label{eq:g2}
	\bg=\bx\cO(\br\bx,\br^2\by,\br^2,\br\ba)+\ba\cO(\br\bx,\br^2\by,\br^2,\br\ba)+\br\by\bar h_6(\bx,\by,\br,\ba),
\end{equation}
where $\bar h_6$ is smooth. Then, it is clear that $g_2\in\cO(\br)$.} Similarly, $\bu=\bu(\bx,\by,\br,\ba)$ is the blown-up state-feedback controller to be designed. Observe that, analogously to what is described in Remark \ref{rem:H}, we have that for $\br=\ba=\bu=0$ the orbits of \eqref{eq:K2-1} are given as level sets of the function
\begin{equation}\label{eq:H2}
	H_2(\bx,\by) = \frac{1}{2}\exp(-2\by)\left(\by-\bx^2+\frac{1}{2}\right).
\end{equation}

Having this in mind, we are going to design $\bu$ in such a way that for a trajectory $(\bx(t_2),\by(t_2))$ of \eqref{eq:K2-1} one has $\lim_{t_2\to\infty} H_2(\bx(t_2),\by(t_2))= h$, where $h$ defines the desired canard cycle and $t_2$ denotes the time-parameter of \eqref{eq:K2-1}.

We approach the design of $\bu$ as follows: we start by restricting to $\left\{ \br=0 \right\}$ and define $\widetilde{H}_2=H_2-h$,
where $h\in(0,\frac{1}{4})$\footnote{In principle our analysis holds for $h\leq\frac{1}{4}$, but only the considered interval provides canard cycles, which are our main focus. See also section \ref{sec:K1}.}. Next we define a candidate Lyapunov function given by
\begin{equation}\label{eq:K2-L}
	L_2(\bx,\by)=\frac{1}{2}\widetilde{H}_2^2,
\end{equation}
and note that $L_2>0$ for all $\widetilde{H}_2\neq0$ and that $L_2=0$ if and only if $\widetilde{H}_2=0$, if and only if $(\bx,\by)\in\gamma_{h}$, where by $\gamma_{h}$ we denote the reference canard cycle, that is
\begin{equation}
	\gamma_{h}=\left\{ (\bx,\by)\in\R^2\,|\, \widetilde{H}_2=0 \right\}.
\end{equation}
It follows that
\begin{equation}\label{eq:K2-L2p}
	L_2'=-\bx\exp(-2\by)\widetilde{H}_2\left( 2\ba\bx + \ba^2 + \bu^0 \right),
\end{equation}
where $\bu^0=\bu(0,\bx,\by,\ba)$. Naturally, we want to design $\bu^0$ such that $L_2'<0$, or at least $L_2'\leq0$. We now see that a convenient choice of $\bu^0$ is 
\begin{equation}\label{eq:K2-u20}
	\bu^0 = -2\ba\bx - \ba^2 + c_1\bx\exp(c_2\by)\widetilde{H}_2,
\end{equation}
where $c_1>0$ and $c_2\in\R$ are the controller gains. Using \eqref{eq:K2-u20} we have
\begin{equation}\label{eq:L2p}
	L_2' = -c_1\bx^2\exp((c_2-2)\by)\widetilde{H}_2^2\leq0.
\end{equation} 
Note that\ch{, because the exponential function is positive,} the previous inequality holds \ch{for every value of} $c_2\in\R$, however a particular choice of $c_2$ may drastically change the performance of the controller\ch{, hence its inclusion in \eqref{eq:K2-u20}}. \ch{This can be readily seen if we substitute $\widetilde{H}_2$ in \eqref{eq:K2-u20}:
\begin{equation}\label{eq:K2-u20-aux}
	\bu^0 = -2\ba\bx - \ba^2 + \frac{1}{2}c_1\bx\exp((c_2-2)\by)\left(\by-\bx^2+\frac{1}{2}\right)-c_1\bx\exp(c_2\by)h.
\end{equation}
Let $D\subset\R^3$ be a bounded domain. We see that $\mu_2^0$ is bounded for all $(\ba,\bx,\by)\in D$. However, since $c_2$ appears inside the exponential, the upper bound of $|\mu_2^0|$ can vary widely depending on the choice of $c_2$.
}
The relevance of $c_2$ shall be detailed in section \ref{sec:K1}. 

By Lasalle's invariance principle \cite{lasalle1960some} we have that, under the controller $\eqref{eq:K2-u20}$ and $\br=0$, the trajectories of \eqref{eq:K2-1} eventually reach the largest invariant set contained in 
\begin{equation}
	\cI=\left\{ (\bx,\by)\in\R^2\,|\, L_2'=0\right\}=\left\{ \bx=0 \right\}\cup\left\{ \widetilde{H}_2=0 \right\}
\end{equation}
Note, however that $\left\{ \bx=0 \right\}$ is generically not invariant for the closed-loop dynamics \eqref{eq:K2-1}. Indeed, the closed-loop system \eqref{eq:K2-1} (restricted to $\br=0$) reads as
\begin{equation}\label{eq:K2-cl1}
	\begin{split}
		\bx' &= -\by + \bx^2 + c_1\bx\exp(c_2\by)\widetilde{H}_2 \\
		\by' &= \bx,
	\end{split}
\end{equation}
where setting $\bx=0$ leads to $(\bx',\by')=(-\by,0)$. Therefore, we now have that all trajectories of \eqref{eq:K2-1}  eventually reach $\cI_2=\left\{ (\bx,\by)=(0,0) \right\}\cup\left\{ \widetilde{H}_2=0 \right\}$. Since the origin is an equilibrium point of \eqref{eq:K2-cl1}\footnote{In fact it is straightforward to further show that the origin is an unstable equilibrium point of \eqref{eq:K2-cl1}.}, we have that every trajectory with initial conditions $(\bx(0),\by(0))\in\R^2\backslash\left\{ (0,0) \right\}$ eventually reaches the set $\left\{\widetilde{H}_2=0\right\}$ as $t_2\to\infty$. With the previous analysis we have shown the following:

\begin{proposition}\label{prop:K2}
	\chh{Consider \eqref{eq:K2-1}. Then, for $\br\geq0$ sufficiently small a controller of the form
	 \begin{equation}\label{eq:K2-u2r}
	 	\begin{split}
		\bu =&  -2\ba\bx-\ba^2 + c_1\bx\exp(c_2\by)\left( H_2-h \right) +\cO(\br),
	\end{split} 
	 \end{equation}
	where $c_1>0$ and $c_2\in\R$ and with $H_2$ is as in \eqref{eq:H2}, renders the orbit $\gamma_{h}$ locally asymptotically stable.}
\end{proposition}
\begin{proof}
	\ch{As described above, the stability of $\gamma_h$ for \eqref{eq:K2-u2r} is equivalent to the stability of the zero solution of
	\begin{equation}\label{H2aux}
		\tilde H_2'=-\bx\exp(-2\by)(2\ba\bx+\ba^2+\mu_2^0)+\cO(\br).
	\end{equation}
	Substituting \eqref{eq:K2-u20} in \eqref{H2aux} we get
	\begin{equation}\label{H2aux2}
		\tilde H_2'= -c_1\bx^2\exp((c_2-\by))\tilde H_2+\cO(\br).
	\end{equation}
	We have shown that for $\br=0$, the origin is locally asymptotically stable for \eqref{H2aux2}. An apparent issue in \eqref{H2aux2} is the term $x_2^2$. However, we have also shown that $\left\{ \bx=0 \right\}$ is not invariant. Therefore \eqref{H2aux2} is a particular case of the non-autonomous scalar equation
	\begin{equation}\label{H2aux3}
		\tilde H_2'=-a(t_2)\tilde H_2+\cO(\br),
	\end{equation}
	where $a(t_2)\geq0$ for all $t_2$ \emph{and} $a(t_2)>0$ for almost all $t_2$ (here $t_2$ is the time parameter in the chart $K_2$). The solution of the unperturbed equation \eqref{H2aux3} is $H_2(t_2)=k \exp\left(-\int_{t_0}^{t_2}a(s_2)\textnormal{d}s_2 \right)$, for some $k\in\R$. So, due to the properties of $a(t_2)$, the trivial solution of \eqref{H2aux3}, with $\br=0$, is asymptotically stable, which is preserved under sufficiently small perturbations $\cO(\br)$ \cite{coppel2006dichotomies}.
	}
\end{proof}

We show in Figure \ref{fig:K2} a simulation of the result postulated in Proposition \ref{prop:K2}. 

\begin{figure}[htbp]\centering
	\begin{tikzpicture}
		\node at (0,0){
			\includegraphics{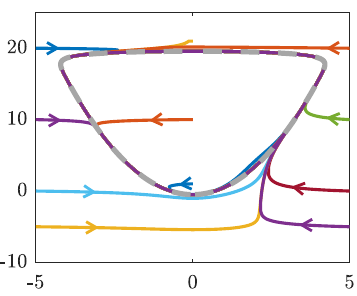}  
		};
		\node at (0.3,-2.75){$\bx$};
		\node at (-3.5,0){$\by$};			
		
	\end{tikzpicture}  
	\caption{Simulation of \eqref{eq:K2-1} with the controller \eqref{eq:K2-u20}. The parameters for the simulation are $(\br,\ba,c_1,c_2,h)=(0,1,1,2,1\times 10^{-16})$. The desired periodic orbit is depicted as the dashed curve.}
	\label{fig:K2}
\end{figure}

Let us emphasize at this point that designing the controller in the rescaling chart justifies using $H_2$ to define a convenient Lyapunov function, even if there are higher order terms in the original vector field \eqref{eq:fast-b}. We also point out that the maximal canard becomes unbounded in this chart. Such a case shall be studied in chart $K_1$ (see section \ref{sec:K1} below). Next we digress on how to deal with a certain class of higher order terms even if $\br$ \ch{(equivalently $\ve$)} is not small. 

\begin{lemma}\label{prop:K2-r} Consider \eqref{eq:K2-1} with $\br>0$ fixed and let $\Gamma_2\subset\R^2$ be a neighbourhood of $\gamma_h$. Assume that the function $g_2$ satisfies
\begin{enumerate}[leftmargin=*]
	\item $g_2=\bx\phi_2(\br,\bx,\by,\ba)$, where $\phi_2$ is smooth \chh{and vanishes at the origin}.
	\item The function $\phi_2$ satisfies \chh{$1+\phi_2(\br,0,\by,\ba)\neq0$} for all $(0,\by)\in\Gamma_2$.
	\item \ch{The function $\phi_2$ satisfies \chh{$1+\phi_2(\br,\bx,\by,\ba)\neq0$} for all $(\bx,\by)\in\gamma_h$}.

\end{enumerate}
Then, \chh{a controller of the form
	 \begin{equation}\label{eq:K2-u2r2}
	 	\begin{split}
		\bu =&  -2\ba\bx-\ba^2 + c_1\bx\exp(c_2\by)\left( H_2-h \right) - (\by-\bx^2)\bphi,
	\end{split} 
	 \end{equation}
	where $c_1>0$ and $c_2\in\R$,} renders $\gamma_{h}$ locally asymptotically stable in $\Gamma_2$.
\end{lemma}
\begin{proof}
	\chh{First, we recall that $g_2\in\cO(\br)$, see \eqref{eq:g2}. Therefore, under the assumptions of the Lemma we can write $g_2=\bx\bphi=\br\bx\bpsi$ for some function $\bpsi$. Next, and}
	similar to the analysis performed above, we consider \eqref{eq:K2-1} but now with an extra $\cO(\br)$-term in the controller, namely
	\begin{equation}
	\begin{split}
		\bx' &= -\by + (\bx+\ba)^2 + \bu^0 + \br\bv \\
		\by' &= \bx + \bg,
	\end{split}
\end{equation}
where $\bu^0$ is as in \eqref{eq:K2-u20} and now $\bv=\bv(\br,\bx,\by,\ba)$ is to be designed. 
Consider, as before, the candidate Lyapunov function \eqref{eq:K2-L}. After substituting $\bu^0$ \chh{and $\bg=\br\bx\bpsi$} we get
\begin{equation}
	L_2'= -c_1\bx^2\widetilde{H}_2^2\exp((c_2-2)\by) - \br\bx\widetilde{H}_2\exp(-2\by)(\bv+(\by-\bx^2)\chh{\bpsi}).
\end{equation}
The above expression suggests to set $\bv=-(y-\bx^2)\chh{\bpsi}$. By doing so one gets \eqref{eq:K2-L2p} again and therefore, invoking again Lasalle's invariance principle, we now take a look at the set $\cI=\left\{ \bx=0 \right\}\cup\left\{ \widetilde{H}_2=0 \right\}$ related to the closed-loop system. To be more precise we now focus on
\begin{equation}\label{eq:K2aux1}
	\begin{split}
		\bx' &= -\by + \bx^2 + c_1\bx\exp(c_2\by)\widetilde{H}_2 - (\by-\bx^2)\bphi \\
		\by' &= \bx + \bx\bphi,
	\end{split}
\end{equation}
\chh{where we have used $\br\bpsi=\bphi$, }and consider its dynamics restricted to $\cI$. On $\left\{ \bx=0 \right\}$ one has $(\bx',\by')=\left(-\by\left(1+\bphi|_{\left\{ \bx=0 \right\}}\right),0\right)$. Therefore, to avoid $\left\{ \bx=0 \right\}$ being invariant we impose the condition $1+\bphi(\br,0,\by,\ba)\neq0$. Note that the aforementioned condition would already suffice to show that trajectories converge towards $\left\{ \widetilde{H}_2=0 \right\}$, however, there may still be a stable equilibrium point contained in \ch{$\left\{ \widetilde{H}_2=0 \right\}$}.  The restriction of \eqref{eq:K2aux1} to $\left\{ \widetilde{H}_2=0 \right\}$ reads as
\begin{equation}\label{eq:K2aux2}
	\begin{split}
		\bx' &= (-\by + \bx^2)(1+ \bphi) \\
		\by' &= \bx(1 + \bphi),
	\end{split}\qquad\qquad (\bx,\by)\in\gamma_{h}.
\end{equation}
Now it suffices to give conditions on $\bphi|_{\left\{(\bx,\by)\in\gamma_{h}\right\}}$ such that \eqref{eq:K2aux2} does not have equilibrium points (keep in mind that $(0,0)\notin\gamma_h$ for $h\in(0,1/4)$). Such a condition is simply \ch{$1+\bphi(\br,\bx,\by,\ba)\neq0$ for all $(\bx,\by)\in\gamma_h$}, completing the proof. 
\end{proof}

\begin{remark}\leavevmode
\begin{itemize}[leftmargin=*]
	\item If the third assumption of Lemma \ref{prop:K2-r} does not hold, then trajectories converge to an equilibrium point contained in the set $\left\{\widetilde{H}_2=0\right\}$. 


	\item \chh{A simpler to check and sufficient condition on $\bphi$ satisfying the hypothesis of Lemma \ref{prop:K2-r} is \chh{$\bphi(\br,\bx,\by,\ba)\neq1$} for all $(\bx,\by)\in\Gamma_2$. \chh{Also, if $g=x\phi(x,y,\ve,\alpha)$ and $g_2=\bx\bphi$ is its blown-up version, then $\bphi(\br,\bx,\by,\ba)=\phi(\br\bx,\br^2\by,\br^2,\br\ba)=\phi(x,y,\ve,\alpha)$. Therefore, $\bphi\neq1$ implies $\phi\neq1$}. These two arguments are the ones we use for Theorem \ref{thm:fold-fast}.}
\end{itemize}
	 
\end{remark}

We show in Figure \ref{fig:K2-r1} a simulation regarding Lemma \ref{prop:K2-r}. 

\begin{figure}[htbp]\centering
	\begin{tikzpicture}	
		\node at (0,0){
			\includegraphics{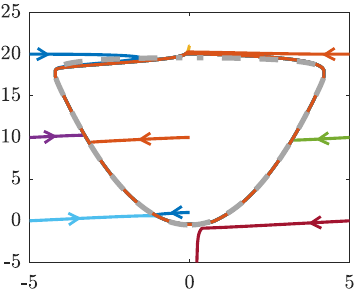}
		};
		\node at (0.3,-2.75){$\bx$};
		\node at (-3.5,0){$\by$};	
	\end{tikzpicture}	\hfill
	\begin{tikzpicture}	
		\node at (0,0){
			\includegraphics{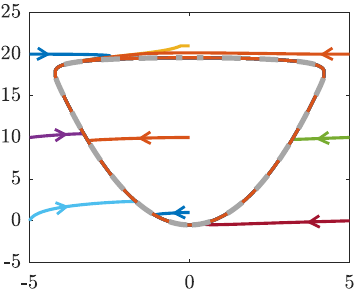}
		};
		\node at (0.3,-2.75){$\bx$};
		\node at (-3.5,0){$\by$};
	\end{tikzpicture}
	\caption{\ch{An example of a} phase portrait corresponding to \eqref{eq:K2-1} with \ch{the particular choice:} $\br=1$, $\ba=1$, $\bphi=\by-\bx^2$ and $(c_1,c_2)=(10,2)$. \ch{Due to the way the local coordinates in this chart are defined, choosing $r_2=1$, essentially amounts to considering $\ve=1$ in \eqref{eq:fast-b}.} On the left we show the orbits corresponding to $\bv=0$, and on the right those for $\bv$ given as in Lemma~\ref{prop:K2-r}. Observe on the left that trajectories do not follow the desired canard while on the right they do. This means that the extra term $\bv$ is suitable to render the canard asymptotically stable when the perturbations of order $\cO(\br)$ in \eqref{eq:K2-1} are not small.}
	\label{fig:K2-r1}
\end{figure}
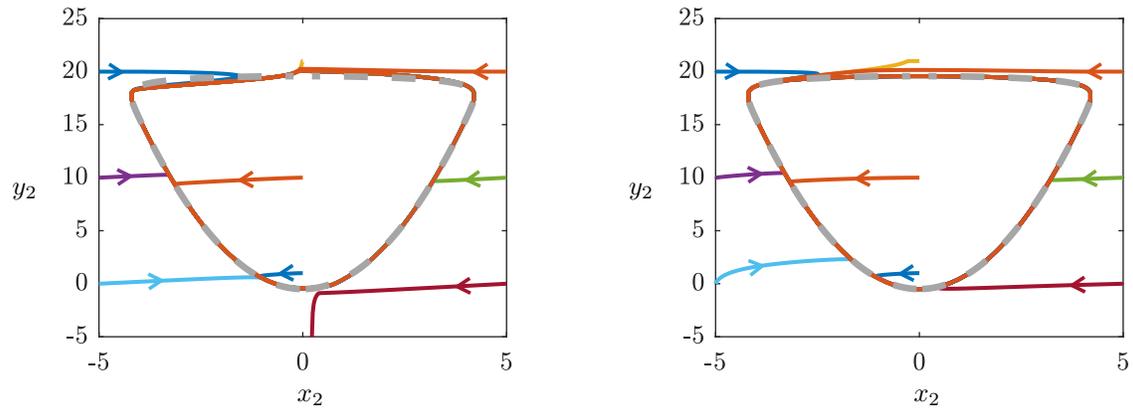

\subsubsection{Analysis in the directional chart \texorpdfstring{$K_1$}{K1}}\label{sec:K1}

\renewcommand{\bx}{x_1}
\renewcommand{\by}{y_1}
\renewcommand{\bz}{z_1}
\renewcommand{\br}{r_1}
\renewcommand{\bu}{\mu_1}
\renewcommand{\bv}{\nu_1}
\renewcommand{\be}{\ve_1}
\renewcommand{\ba}{\alpha_1}
\renewcommand{\bg}{g_1}
\renewcommand{\bphi}{\phi_1}

We are now going to look at the controlled dynamics in the chart $K_1$. This serves two purposes: the first is of giving a more precise meaning to the constant $c_2$ in the controller \eqref{eq:K2-u2r}; the second is to corroborate that the controller designed previously is indeed able to also stabilize the (unbounded) maximal canard. \ch{Using the definition on $K_1$ as in \eqref{eq:Ks}, we have that} the dynamics in this chart read as

\begin{equation}\label{eq:K1}
	\begin{split}
		\br' &= \frac{1}{2}\br\be\bx+\cO(\br\be)\\
		\bx' &= -1+(\bx+\ba)^2-\frac{1}{2}\be\bx^2+\bu+\cO(\br\be)\\ 
		\be' &=-\be^2\bx+\cO(\br\be)\\
		\ba' &=-\frac{1}{2}\be\bx+\cO(\br\be),
	\end{split}
\end{equation}%
\ch{where, in particular, $\mu_1$ denotes the controller written in the local coordinates of this chart. Since we have already designed a controller in the chart $K_2$, see \eqref{eq:K2-u2r}, we can use the transformation \eqref{eq:match21} to express $\mu_1$ as
\begin{equation}\label{eq:K1u}
	\begin{split}
		\bu =\kappa_{21}(\mu_2)&=-2\ba\bx-\ba^2+c_1\be^{1/2}\bx\exp(c_2\be^{-1})\widetilde{H}_1+\cO(\br\be^{3/2}),
	\end{split}
\end{equation}
where, analogous to what we have done in chart $K_2$, we define $\widetilde{H}_1=H_1-h$ with
\begin{equation}\label{eq:H1}
	\begin{split}
		H_1 = \kappa_{21}(H_2) &= \frac{1}{2}\exp\left(-2\be^{-1}\right)\left(\be^{-1}-\be^{-1}\bx^2+\frac{1}{2}\right).
	\end{split}
\end{equation}
}

\begin{remark}\leavevmode
\begin{itemize}[leftmargin=*]
	\item $\bu$ is bounded along any reference canard $\gamma_h=\left\{ \widetilde{H}_1=0\right\}$ with $h\in(0,\frac{1}{4})$.
	\item If $h\neq0$, then $\bu$ becomes unbounded as $\be\to0$ unless $\widetilde{H}_1=0$ (previous observation). This is to be expected as, in the limit $\be\to0$ the only canard orbit to stabilize is the maximal canard since $\lim_{\be\to0}H_1=0$. Therefore, we are going to study the closed-loop dynamics \eqref{eq:K1} for the particular choice of $h=0$ and for the limit $\be\to0$. \ch{Our goal is to refine the constant $c_2$ so that $\mu_1$ remains bounded whenever $h=0$ and $\be\to0$. Moreover, recalling that for this chart we have $\be=\frac{y}{\ve}$, the limit $\be\to0$ corresponds to the limit $y\to\infty$ for fixed $\ve>0$.}
\end{itemize}
	
\end{remark}

So from now on we let $h=0$, that is $\widetilde{H}_1=H_1=\frac{1}{2}\exp\left(-2\be^{-1}\right)\left(\be^{-1}-\be^{-1}\bx^2+\frac{1}{2}\right)$. We also restrict to  $\left\{\br=0 \right\}$. In such a case we have
\begin{equation}\label{eq:K1u1}
	\bu=-2\ba\bx-\ba^2+\frac{1}{2}c_1\be^{1/2}\bx\exp\left((c_2-2)\be^{-1}\right)\left(\be^{-1}-\be^{-1}\bx^2+\frac{1}{2}\right),
\end{equation}
and the closed loop system reads as
\begin{equation}
	\begin{split}
		\bx' &= -1+\bx^2-\frac{1}{2}\be\bx^2+\frac{1}{2}c_1\be^{1/2}\bx\exp(c_2\be^{-1})H_1\\ 
		\be' &=-\be^2\bx\\
		\ba' &=-\frac{1}{2}\be\bx.
	\end{split}
\end{equation}
It shall also be relevant to consider $H_1'$, namely
\begin{equation}\label{eq:H1p}
\begin{split}
	H_1'&=-\frac{1}{2}c_1\be^{-1/2}\bx^2\exp\left( (c_2-2)\be^{-1}\right)H_1\\
	&=-\frac{1}{2}c_1\be^{-1/2}\bx^2\exp\left( (c_2-4)\be^{-1}\right)\left(\be^{-1}-\be^{-1}\bx^2+\frac{1}{2}\right).
\end{split}
\end{equation}

First of all we note that $\lim_{\be\to0} H_1=0$, and $\lim_{\be\to0} H_1'=0$ for $c_2<4$. Next, we focus on \eqref{eq:K1u1} where we observe that in order for the controller to be bounded as $\be\to0$the constant $c_2$ should be less than $2$. To be more precise:

\ch{
\begin{lemma}\label{lemma:K1_c2}
	Let $(\ba,\bx)$ be bounded and $c_1>0$. Then $\lim_{\be\to0}|\bu|<\infty$ if and only if $c_2<2$.
\end{lemma}
\begin{proof}
	Straightforward computations.
\end{proof}
From Lemma \ref{lemma:K1_c2} we have that, to follow the maximal canard ($h=0$) one must choose $c_2<2$ to ensure that the controller is bounded. Although analytically any choice of $c_2<2$ suffices, a particular choice may influence drastically numerical simulations since $c_2$ appears in the exponential. For instance, we see from the first line of \eqref{eq:H1p} that $c_2<2$ but arbitrarily close to $2$ reduces the contribution of the exponential term, which may induce issues in numerical simulations. 
For all other canards, $c_2\in\R$ is sufficient. However, again from the computational perspective, $c_2=2$ is the appropriate choice as it eliminates the exponential term in \eqref{eq:K1u1} and in \eqref{eq:H1p}, which is rather convenient for simulations. We remark that a completely analogous analysis, which we omit for brevity, follows for the chart $K_3=\left\{ \bar x=1 \right\}$ where canards corresponding to $h<0$ can be considered. The arguments and the conclusion are the same, namely, for $h<0$ one should set $c_2<2$ so that the controller remains bounded along the unbounded canards.
}

\subsubsection{Proof of Theorem \ref{thm:fold-fast}}\label{proof:fold-fast}
\renewcommand{\bx}{x_2}
\renewcommand{\by}{y_2}
\renewcommand{\bz}{z_2}
\renewcommand{\br}{r_2}
\renewcommand{\bu}{\mu_2}
\renewcommand{\bv}{\nu_2}
\renewcommand{\be}{\ve_2}
\renewcommand{\ba}{\alpha_2}
\renewcommand{\bg}{g_2}
\renewcommand{\bphi}{\phi_2}
\newcommand{\hx}{\hat x}

To prove Theorem \ref{thm:fold-fast} we first blow-down the controller $\bu$. To keep it simple we shall blow-down \eqref{eq:K2-u2r}, but of course the same holds for \eqref{eq:K2-u2r2}. So, recall from \eqref{eq:K2-u2r} that the blown-up controller is
\begin{equation}
	\bu=-2\ba\bx-\ba^2 + c_1\bx\exp(c_2\by)\left( H_2-h \right).
\end{equation}
Next, from \eqref{eq:Ks} we have
\begin{equation}\label{eq:u_proof}
	\hat u=\ve \bu = -2\alpha\hx - \alpha^2 + c_1\hx\ve^{1/2}\exp\left(c_2y\ve^{-1}\right)(\hat H-h),
\end{equation}
where $\hat H=\hat H(\hx,y,\ve)=\frac{1}{2}\exp\left( -\frac{2y}{\ve}\right)\left( \frac{y}{\ve}-\frac{\hx^2}{\ve}+\frac{1}{2} \right)$ as stated in the first item of Theorem \ref{thm:fold-fast}. Under \eqref{eq:u_proof} the closed-loop system corresponding to \eqref{eq:fast-b} reads as
\begin{equation}\label{eq:cl_proof}
	\begin{split}
		\hx' &= -y+\hx^2 + c_1\hx\ve^{1/2}\exp\left(c_2y\ve^{-1}\right)(\hat H-h)\\
		y' &= \ve(\hx+\widetilde g).
	\end{split}
\end{equation}
Next, it is important to observe that $\lim_{\ve\to0}\hat H=0$. This means that for $\ve=0$ the only reference canard that is reachable is the maximal canard\footnote{As it is expected, the controller becomes unbounded in the limit $\ve\to0$ for any other canard, as they do not exist in such a limit.}. The maximal canard corresponds to $h=0$. So, setting $h=0$, and since one chooses $c_2<2$ (recall Section \ref{sec:K1}), it follows that $\lim_{\ve\to0}c_1\hx\ve^{1/2}\exp\left(c_2y\ve^{-1}\right)\hat H=0$, meaning that the layer equation for \eqref{eq:cl_proof} is 
\begin{equation}
	\begin{split}
		\hx' &= -y+\hx^2 \\
		y' &= 0,
	\end{split}
\end{equation}
which indeed has the same type of singularity at the origin as the open-loop system, a fold. This shows that \eqref{eq:u_proof} is a compatible controller in the sense of Definition \ref{def:comp}.

\subsection{The slow control problem}\label{sec:slow}

In this section we consider the slow-control problem
\begin{equation}\label{eq:slow-thm}
	\begin{split}
		x' &= -y+x^2+ f(x,y,\ve,\alpha) \\
		y' &= \ve(x-\alpha+u(x,y,\ve,\alpha)),
	\end{split}
\end{equation}
where the objective is, as in Section \ref{sec:fast}, to stabilize a prescribed canard $\gamma_h$. Due to space constraints, and because the analysis is similar to the one performed in Section \ref{sec:fast}, we only state the relevant result.

\begin{theorem}\label{thm:slow}Consider \eqref{eq:slow-thm} and let $\hat H=H(x,y,\ve)$ be defined by \eqref{eq:H}. Then, \chh{the compatible controller
\begin{equation}\label{eq:thmslow_u1}
	 u=\alpha+c_1(y-x^2)\ve^{-1/2}\exp(c_2y\ve^{-1})(H-h),
\end{equation}
where $c_1>0$, $c_2\in\R$ and $h\leq\frac{1}{4}$} renders the canard orbit $\gamma_h=\left\{ (x,y)\in\R^2\, | \, H=h \right\}$ locally asymptotically stable for $\ve>0$ sufficiently small. A convenient choice of controller gain $c_2$ for the maximal canard is \ch{$c_2<2$}. By convenient we mean that such a choice ensures that the controller remains bounded as $y\to\infty$.
\end{theorem}

In Figure \ref{fig:fold-slow} we illustrate the statement of Theorem \ref{thm:slow}.
\newpage

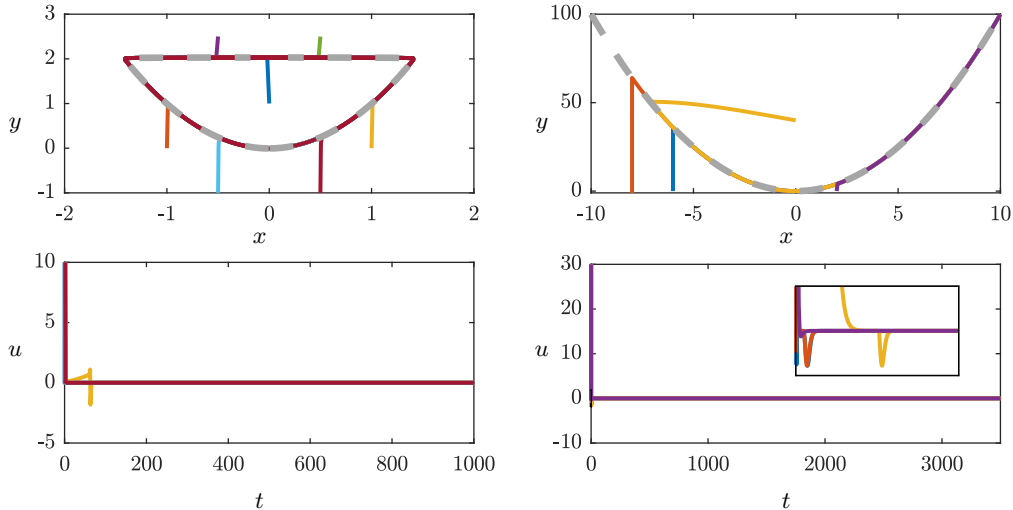
\begin{figure}[htbp]\centering
	\begin{tikzpicture}
		\node at (0,0){
		\begin{tikzpicture}
			\node at (0,0){
				\includegraphics{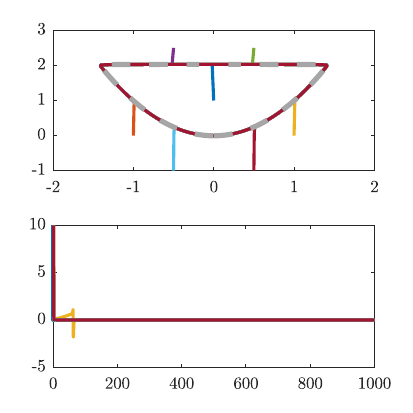}  
			};
			\node at (0,-3.5){\small$t$};
			\node at (0,0){\small$x$};
			\node at (-3.25,1.5){\small$y$};
			\node at (-3.25,-1.5){\small$u$};
		\end{tikzpicture}
		};
		\node at (7,0){
		\begin{tikzpicture}
			\node at (0,0){
				\includegraphics{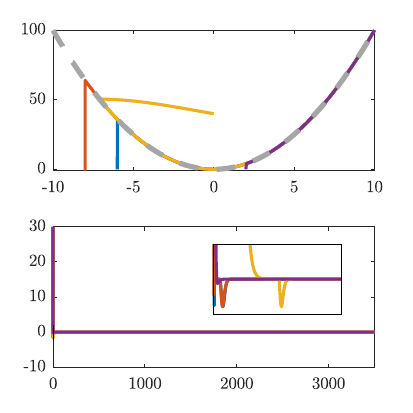}
			};
			\node at (0,-3.5){\small$t$};
			\node at (0,0){\small$x$};
			\node at (-3.25,1.5){\small$y$};
			\node at (-3.25,-1.5){\small$u$};
		\end{tikzpicture}
		};
	\end{tikzpicture}
	\caption{ The first column corresponds to the control of a bounded canard cycle (shown in dashed-gray), while the second column to the control of the maximal canard. The first row shows the phase-portrait in $(x,y)$-coordinates. The second row shows the time series of the corresponding controller. We show on the lower-right diagram a detail of the controller's signal for time close to $0$. }
	\label{fig:fold-slow}
\end{figure}


\section{Controlling Canard Cycles for the van der Pol oscillator}\label{sec:vdp}

In this section we are going to extend the ideas developed previously to control canard cycles in the van der Pol oscillator. \ch{The main idea is to adapt and extend the controller proposed in Theorem \ref{thm:fold-fast}, and to use it to control canard cycles of the van der Pol oscillator. In this context we distinguish two types of canard cycles: a) \emph{canards with head} and b) \emph{canards without head}. Canards with head refer to canard cycles with two fast segments, while canards without head have only one fast segment, see Figure \ref{fig:sims}}. Furthermore, due to its relationship with some neuron models, like the Fitzhugh-Nagumo model \cite{fitzhugh1969mathematical,nagumo1962active}, we shall consider that the controller acts on the fast variable only. The idea is that the controller represents input current. Thus, let us study
\begin{equation}\label{eq:vdp}
	\begin{split}
		x'&=-y+x^2-\frac{1}{3}x^3+u\\
		y'&=\ve x.
	\end{split}
\end{equation}
\begin{remark}
	For simplicity, we have chosen to present the case $\alpha=0$. However, the case $\alpha\neq0$ follows straightforwardly from considering the arguments at the beginning of Section \ref{sec:fast}.
\end{remark}

\ch{The corresponding critical manifold reads as,
\begin{equation}
	\cS_0 = \left\{ (x,y)\in\R^2\, : \, y=x^2-\frac{1}{3}x^3\right\}.
\end{equation}
The repelling and attracting parts of $\cS_0$ are denoted respectively by $\cS_0^\txtr$ and $\cS_0^\txta$, and are given by
\begin{equation}
	\begin{split}
		\cS_0^\txtr &= \left\{ (x,y)\in\cS_0\, : \, 0<x<2\right\}\\
		\cS_0^\txta &= \left\{ (x,y)\in\cS_0\, : \, x<0\right\}\cup\left\{ (x,y)\in\cS_0\, : \, x>2\right\}.
	\end{split}
\end{equation}
}

Furthermore, system \eqref{eq:vdp} has two fold points, one at the origin and one at $(x,y)=(2,\frac{4}{3})$. In fact, the origin is a canard point and the singular limit of \eqref{eq:vdp} is as shown in Figure \ref{fig:vdp1}.

To state our main result, \ch{let $\cN_1\subset\R^2$ be a region containing a subset of the repelling critical manifold $\cS_0^\txtr$ and $\cN_2\subset\R^2$ a small region containing a subset of $\cS_0$ around the origin. Although it is not necessary to be precise on such regions, since several choices are possible, an example of $\cN_1$ and $\cN_2$ is as follows}
\begin{equation}\label{eq:Ns}
	\begin{split}
		\cN_1 &= \left\{ (x,y)\in\R^2 \, : \, |-y+x^2-\frac{1}{3}x^3|<\beta_1,\, 0<x<2,\, y_{\txtmin}<y<y_{h} \right\}\\
		\cN_2 &= \left\{ (x,y)\in\R^2 \, : \, |-y+x^2|<\beta_2, \, -x_\txtmin<x<x_{\txtmax} \right\},
	\end{split}
\end{equation}
where \ch{the defining positive constants are such that $\cN_1$ and $\cN_2$ have a non-empty intersection in the first quadrant, and} $0<y_{\txtmin}\in\cO(\ve)$ and $y_{\txtmin}<y_h<\frac{4}{3}$. \ch{The precise meaning of these bounds is given in sections \ref{sec:K1vdp} and \ref{sec:composite}, and is already sketched in Figure \ref{fig:vdp1}.}

\begin{proposition}\label{thm:canardsvdp}
	Consider \eqref{eq:vdp}, let $\psi_i$ be a bump function with support $\cN_i$, and let the repelling slow manifold $\cS_\ve^\txtr$ be given by the graph of $x=\phi(y,\ve)$. Then, one can choose $\cN_i$,  positive constants $c_1$ and $k_1$, and a small constant $x^*$, $|x^*|\ll1$, such that the controller
	\begin{equation}
		\begin{split}
			u &= \frac{1}{2}u_1\psi_1 + \frac{1}{2}u_2\psi_2,
		\end{split}
	\end{equation}
	where 
	\begin{equation}
		\begin{split}
			u_1 &= - F_0 - F_{x^*} + v_1\\
			u_2 &= c_1x\ve^{-1/2}\left(y-x^2+\frac{\ve}{2}\right),
		\end{split}
	\end{equation}
	and with
	\begin{equation}
		\begin{split}
			F_{x^*}(x,y,\ve) &= -y + (x-x^*\sqrt{y})^2 - \frac{(x-x^*\sqrt{y})^2\ve}{2y} - \frac{1}{3}(x-x^*\sqrt{y})^3\\
			v_1(y,\ve) &= \frac{ 2\phi +x^*\sqrt{y} }{\phi}\left( -y + \phi^2-\frac{\ve}{2y}\phi^2-\frac{1}{3}\phi^3 \right) - \left( \frac{\ve}{y}\phi+\sqrt{y}\phi^2+k_1\sqrt{y}\right)(x-\phi-x^*\sqrt{y}),
		\end{split}
	\end{equation}
	stabilizes a canard cycle with height $y_h$. Moreover, if $x^*<0$ then the canard is without head, while if $x^*>0$ then the canard is with head.
\end{proposition}

\begin{proof}[Sketch of proof:] As before, all the analysis is carried-out in the blow-up space.  The overall idea is as follows: the controller to be designed acts only within a small neighbourhood of $\left\{ 0 \right\}\cup\cS_\ve^\txtr$, mainly because the rest of the slow manifold is already stable, so there is no need of stabilizing it. The desired height of the canard is regulated by the constant $y_h$. \ch{The controller $u_2$ controls the trajectories near the canard point and therefore is given by Theorem \ref{thm:fold-fast}, where we have made the choice $h=0$ and $c_2=2$.} So, the new analysis is performed in the chart $K_1=\left\{ \bar y=1 \right\}$, where the objective is to stabilize the (normally hyperbolic) repelling branch of the slow manifold $\cS_\ve^\txtr$ \ch{resulting in the controller $u_1$. Later, in section \ref{sec:composite}, we combine the two controllers and justify the form of the controller given in the Proposition}. The most important feature \ch{of $u_1$} is to control the location of the orbits relative to $\cS_\ve^\txtr$ as it is precisely such location that determines the direction of the jump once the orbits reach the desired height. To avoid smoothness issues, the regions where the controllers are active are defined via bump functions. A schematic representation of this idea is provided in Figure \ref{fig:vdp1}, while the details of the proof follows from Sections \ref{sec:K1vdp} and \ref{sec:composite}. 
\end{proof}

\begin{figure}[htbp]\centering
	\begin{tikzpicture}
		\node at (0,0){
			\includegraphics[scale=1.5]{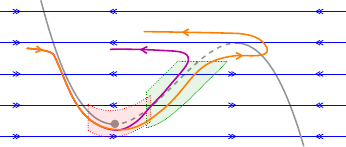}
		};
		\node at (.7,.6) {\color{gray}{\small$\cS_0^\txtr$}};
	\end{tikzpicture}
	\caption{Strategy for the control design: first, within a small neighborhood of the canard point (red-shaded region), we use the controller designed in Section \ref{sec:fold}. Afterwards, a second controller is designed in chart $K_1$ and whose task is to stabilize the (normally hyperbolic) repelling branch $\cS_\ve^\txtr$. This second controller is active on a neighborhood of $\cS_0^\txtr$ (green-shaded region). Furthermore, it is via such controller that we steer the orbits towards either side of $\cS_\ve^\txtr$. This induces that the trajectories jump towards the desired direction once the second controller is inactive. The two orbits illustrate the aforementioned strategy.}
	\label{fig:vdp1}
\end{figure}

In Figure \ref{fig:sims} we show some simulations using the proposed controller.

\begin{figure}[htbp]\centering
	
	\begin{tikzpicture}
		\node (noHead) at (-3.85,0){
			\includegraphics[scale=1]{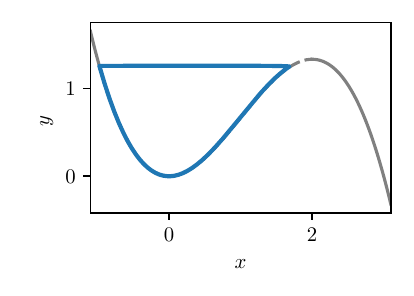}
		};
		\node (withHead) at (3.85,0){
			\includegraphics[scale=1]{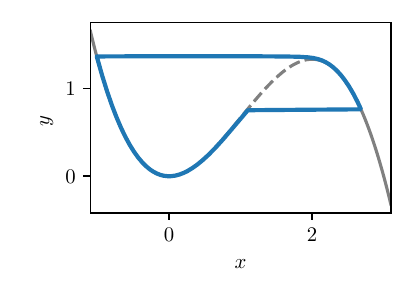}
		};
		\node[text width=5.5cm, align=justify] at (-4,-3){
		(a) A stable canard without head. For this simulation we have set $y_h=1.25$ and $x^*=-0.01$.
		};
		\node[text width=5.5cm, align=justify] at ( 4,-3){
		(b) A stable canard with head. For this simulation we have set $y_h=0.75$ and $x^*=0.01$.
		};
	\end{tikzpicture}

	\caption{Numerical simulations illustrating the controller of Proposition \ref{thm:canardsvdp}. In all these simulations we have used $\ve=0.01$ and $(c_1,k_1)=(1,1)$.}
	\label{fig:sims}
\end{figure}

Before proceeding with the proof of Proposition \ref{thm:canardsvdp}, let us point out that it is straightforward to use the proposed controller to produce robust mixed mode oscillations (MMOs) \cite{desroches2012mixed}. One way to do this is as follows: first of all, we assume that we are able to count the number of small amplitude oscillations (SAOs) and of large amplitude oscillations (LAOs). Next, let us say that we start by following a canard without head, so we set the controller constant $x^*<0$ and $y_h$ to the desired height. After the number of desired SAOs has been reached, we change the controller constant $x^*$ to $x^*>0$ and, if desired, $y_h$ to a new height value. So, the controller will now steer the system to follow a canard with head. This process can be repeated to produce any other pattern allowed by the geometry of the van der Pol oscillator. We show in Figure \ref{fig:vdp-mmos} an example of stable MMOs that are obtained using the controller of Proposition \ref{thm:canardsvdp}.
\begin{figure}[htbp]\centering
	\begin{tikzpicture}
		\node at (-4,0){
			\includegraphics[scale=1]{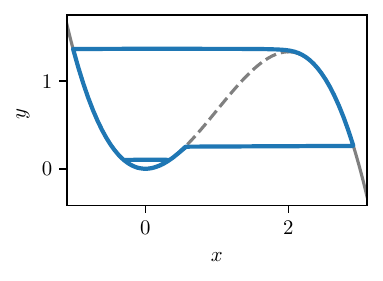}
		};
		\node at (4,2){
			\includegraphics[scale=1]{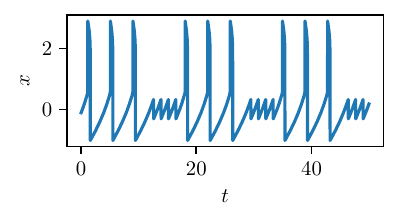}
		};
		\node at (4,-2){
			\includegraphics[scale=1]{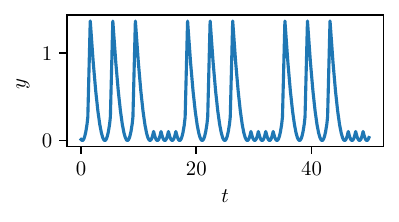}
		};
	\end{tikzpicture}
	\caption{A sample of a Mixed-Mode Oscillation (MMO) with 3 Large Amplitude Oscillations (LAOs) and 4 Small Amplitude Oscillations (SAOs) obtained using the controller of Proposition \ref{thm:canardsvdp}.}
	\label{fig:vdp-mmos}
\end{figure}

\subsection{Analysis in the directional chart \texorpdfstring{$K_1$}{K1}}\label{sec:K1vdp}

\renewcommand{\bx}{x_1}
\renewcommand{\by}{y_1}
\renewcommand{\bz}{z_1}
\renewcommand{\br}{r_1}
\renewcommand{\bu}{\mu_1}
\renewcommand{\bv}{\nu_1}
\renewcommand{\be}{\ve_1}
\renewcommand{\ba}{\alpha_1}
\renewcommand{\bg}{g_1}
\renewcommand{\bphi}{\phi_1}

Similar to the analysis in section \ref{sec:K1} we use a directional blow-up defined by
\begin{equation}\label{eq:vdpK1blowup}
	x=\br\bx, \quad y=\br^2, \quad\quad \ve=\br^2\be, \quad u=\br^2\bu, \quad \alpha=\br\ba.
\end{equation}
Therefore, the local vector field associated to \eqref{eq:vdp} reads as
\begin{equation}\label{eq:K1vdp}
	\begin{split}
		\br' &= \frac{1}{2}\br\be\bx\\
		\be' &= -\be^2\bx\\
		\bx' &= -1+\bx^2-\frac{1}{2}\bx^2\be-\frac{1}{3}\br\bx^3+\bu.
	\end{split}
\end{equation}
To have a better idea of what we are going to achieve with the controller, it is worth to first look at the uncontrolled dynamics. 

Let us define a domain
\begin{equation}
	D_1 = \left\{ (\br,\be,\bx)\in\R^3\, |\,  0\leq\br\leq\rho_1, \; 0\leq\be\leq\delta_1, \; \bx\in\R \right\}.
\end{equation}

\begin{lemma}\label{lemm:K1-dynamics} Consider \eqref{eq:K1vdp} with $\bu=0$. Then, one can choose constants $\rho_1>0$ and $\delta_1>0$ such that the following properties hold within the domain $D_1$.
\begin{enumerate}[leftmargin=*]
	\item There exist semi-hyperbolic equilibrium points $p_{1,\pm}=(\br,\be,\bx)=(0,0,\pm1)$. The point $p_{1,-}$ is attracting while $p_{1,+}$ is repelling along the $\bx$-axis.
	\item Let $\cM_1$ be defined by
	\begin{equation}
		\cM_1=\left\{ (\br,\be,\bx)\in\R^3\;|\; \be=0,\, \br=3\left(\frac{1}{\bx}-\frac{1}{\bx^3}\right) \right\}.
	\end{equation}
	The set $\cM_1$ corresponds to the set of equilibrium points of \eqref{eq:K1vdp} restricted to $\left\{ \be=0 \right\}$. Moreover, let us denote the subsets
	\begin{equation}
		\begin{split}
			\cM_{1,-} &= \left\{ (\br,\be,\bx)\in\cM_1 \,|\, \bx<0 \right\},\\
			\cM_{1,+} &= \left\{ (\br,\be,\bx)\in\cM_1 \,|\, \bx\geq1\right\}.\\
		\end{split}
	\end{equation}
	The subset $\cM_{1,-}$ is attracting and the subset $\cM_{1,+}$ can be partitioned as $\cM_{1,+}=\cM_{1,+}^\txtr\cup\left\{ \left(\frac{3-\sqrt{3}}{\sqrt{3}},0,\sqrt{3}\right)\right\}\cup\cM_{1,+}^\txta$ where
	\begin{equation}
	\begin{split}
		\cM_{1,+}^\txtr &= \left\{ (\br,\be,\bx)\in\cM_{1,+} \,|\, 1\leq\bx<\sqrt{3} \right\},\\
		\cM_{1,+}^\txta &= \left\{ (\br,\be,\bx)\in\cM_{1,+} \,|\, \bx>\sqrt{3} \right\}
	\end{split}
	\end{equation}
	are the repelling and attracting branches of $\cM_{1,+}$, respectively. 

	\item Restricted to $\left\{ \br=0 \right\}$ there exist $1$-dimensional local center manifolds $\cE_{1,-}$ and $\cE_{1,+}$ located, respectively, at the points $p_{1,-}$ and $p_{1,+}$. Such manifolds are given by
	\begin{equation}\label{eq:h1}
		\begin{split}
			\cE_{1,\pm} = \left\{ (\br,\be,\bx)\in\R^3\;|\; \br=0,\, \bx=h_{1,\pm}(\be) \right\},
		\end{split}
	\end{equation}
	where
	\begin{equation}
		h_{1,\pm}(\be)=\pm\left(1+\frac{\be}{2}\right)^{1/2}.
	\end{equation}
	The flow along $\cE_{1,-}$ is directed away from the point $p_{1,-}$ and the flow along $\cE_{1,+}$ is directed towards the point $p_{1,+}$. Furthermore, the center manifolds $\cE_{1,\pm}$ are unique.
	\item There exist $2$-dimensional local centre manifolds $\cW_{1,\pm}$ that contain, respectively, the point $p_{1,\pm}$, the branch of equilibrium points $\cM_{1,\pm}$ and the centre manifold $\cE_{1,\pm}$. These centre manifolds are unique and, moreover, $\cW_{1,-}$ is attracting and  $\cW_{1,+}$ is repelling.

\end{enumerate}
\end{lemma}

\begin{proof}[Sketch of the proof, see \cite{krupa2001relaxation} for details]
	The first two items are obtained by straightforward computations. The expression of the centre manifolds follow from the fact that the restriction of \eqref{sec:K1vdp} to $\left\{ \br=0 \right\}$ has the invariant (just as in the fold case) $H_1 = \frac{1}{2}\exp\left(-2\be^{-1}\right)\left(\be^{-1}-\be^{-1}\bx^2+\frac{1}{2}\right)$. Therefore the functions $h_{1,\pm}$ are given by the solutions of $H_1=0$. The flow on $\cE_{1,\pm}$ follows from the equation $\be'=-\be^2\bx$. The uniqueness of $\cE_{1,\pm}$ is due to the fact that $p_{1,\pm}$ is a semi-hyperbolic saddle of the dynamics of \eqref{eq:K1vdp} restricted to $\left\{\br=0\right\}$. Finally, the existence and properties of $\cW_{1,\pm}$ follow from local analysis at $p_{1,\pm}$, centre manifold theory, the previous arguments, and by choosing $\rho_1<\frac{3-\sqrt{3}}{\sqrt{3}}$. The previous choice of $\rho_1$ is particularly necessary for the stability property of  $\cW_{1,+}$.
\end{proof}

\begin{remark}
	$\cW_{1,+}$ is related, via the blow-up map, to $\cS_\ve^\txtr$. Therefore, the task of the controller is going to be to stabilize the centre manifold $\cW_{1,+}$.
\end{remark}

\begin{remark}
	In what follows, we are going to define some geometric objects, in particular centre manifolds, for the closed-loop dynamics. To make a clear distinction between their open-loop counterparts, and to be able to compare them, we shall denote relevant geometric objects of the closed-loop system by a $\cl$ superscript. 
\end{remark}

In this section we are going to be interested only in $(\bx,\be)\in\R^2_+$. So, to simplify notation let
\newcommand{\hcl}{h_1^{\cl}}
\begin{equation}
	\hcl(\be)=h_{1,+}(\be)=\left( 1+\frac{\be}{2}\right)^{1/2}.
\end{equation}
Furthermore, let us assume that the centre manifold $\cW_{1,+}$ (recall Lemma \ref{lemm:K1-dynamics}) is given by the graph of
\begin{equation}
	\bx=\phi_1(\br,\be).
\end{equation}
	
Note that $\phi_1(0,\be)=\hcl(\be)$. Therefore, one can in fact write
\begin{equation}\label{eq:phi1}
	\phi_1(\br,\be)=\hcl(\be)+\sum_{\substack{i\geq1\\j\geq0}}\sigma_{ij}\br^i\be^j,
\end{equation}
for some coefficients $\sigma_{ij}\in\R$. We now proceed with the following steps.

\begin{enumerate}[leftmargin=*]
	\item \textit{Reverse the direction of the flow in the $\bx$-direction:} Define $f_1(\br,\be,\bx)=-1+\bx^2-\frac{1}{2}\bx^2\be-\frac{1}{3}\br\bx^3$ and let $\bu=-f_1(\br,\be,\bx)-f_1(\br,\be,\bx-\bx^*)+\bv$, where $\bx^*\sim0$ is a constant (the usefulness of $\bx^*$ will become evident below) and $\bv=\bv(\br,\be,\bx)$ is to be further designed. With this step we have that \eqref{eq:K1vdp} now reads as
	\begin{equation}\label{eq:K1vdp2}
		\begin{split}
			\br' &= \frac{1}{2}\br\be\bx\\
			\be' &= -\be^2\bx\\
			\bx' &= -f_1(\br,\be,\bx-\bx^*)+\bv.
		\end{split}
	\end{equation}

	\item \textit{Design $\bv$ so that \eqref{eq:K1vdp2} has $\cW_1^\cl\coloneqq \left\{ (\br,\be,\br)\in\R^3\,|\,\bx=\bx^*+\phi_1(\br,\be)\right\}$ as a closed-loop centre manifold}: this step requires standard centre manifold computations. By performing them we find that
	\begin{equation}
	\begin{split}
		\bv &= \frac{2\phi_1+\bx^*}{\phi_1}\left( -1+\phi_1^2-\frac{1}{2}\phi_1^2\be-\frac{1}{3}\br\phi_1^3 \right).
	\end{split}
	 \end{equation} 

	Note that, if we restrict to $\left\{ \br=0 \right\}$, \eqref{eq:K1vdp2} now reads as
	\begin{equation}\label{eq:K1vdp3}
		\begin{split}
			\be' &= -\be^2\bx\\
			\bx' &= 1-(\bx-\bx^*)^2 + \frac{1}{2}(\bx-\bx^*)^2\be + \be^2\frac{2\hcl+\bx^*}{4\hcl}.
		\end{split}
	\end{equation}

	We know that \eqref{eq:K1vdp3} has a centre manifold $\cE_1^\cl\coloneqq \left\{ (\be,\bx)\in\R^2_{\geq0}\,|\,\bx=\bx^*+\hcl(\be)\right\}$. Furthermore, it follows from straightforward computations that the equilibrium point $p_1^*\coloneqq (0,0,1+\bx^*)$ is attracting along the $\bx$-axis. This means that $\cE_1^\cl$, and also $\cW_1^\cl$, are locally (near $p_1^*$) attracting. Next we improve such stability.

	\item \textit{Design a variational controller that renders $\cW_1^\cl$ locally exponentially stable:} For this it is enough to take the $\bx$-component of the variational equation. So, let $z_1=\bx-\phi_1-\bx^*$. The corresponding variational equation along $\cW_1^\cl$ is
	\begin{equation}\label{eq:K1var}
		\begin{split}
			z_1' &= (-2+\be+\br\phi_1)\phi_1 z_1.
		\end{split}
	\end{equation}

	Recall from \eqref{eq:phi1} that $\phi_1>0$ for $\br\geq0$ sufficiently small. Then, we propose to introduce in \eqref{eq:K1var} a variational controller $w_1(\be,z_1)$ of the form
	\begin{equation}
		w_1=-(\be\phi_1+\br\phi_1^2+k_1)z_1,
	\end{equation}
	where $k_1\geq0$. With $w_1$ as above, the closed-loop variational equation becomes
	\begin{equation}
		z_1' = -\left(2\phi_1 + k_1 \right)z_1,
	\end{equation}
	and we readily see that, for $\br\geq0$ sufficiently small, $z_1\to0$ exponentially as $t_1\to\infty$ (where by $t_1$ we denote the time-parameter in this chart). We also notice that the constant $k_1$ helps to improve the contraction rate towards $\cW_1^\cl$.
	Moreover, since $w_1$ vanishes along $\cW_1^\cl$, the variational controller does not change the closed-loop centre manifold $\cW_1^\cl$. 
	Finally, observe that the role of the small constant $\bx^*$ is to shift the position of $\cW_{1}^\cl$ relative to its open-loop counterpart $\cW_{1,+}$. This is important in order to tune the direction along which the trajectory jumps once the controller is inactive.

	\item \textit{Restrict next to $\left\{ \be=0\right\}$:} Note that $\bv(\br,0,\bx)=0$, then we have
	\begin{equation}\label{eq:K1vdpr0}
		\begin{split}
			\br' &= 0\\
			\be' &= 0\\
			\bx' &= -f_1(\br,0,\bx-\bx^*).
		\end{split}
	\end{equation}

	Similar to the previous step, the new line of equilibrium points is slightly shifted according to $\bx^*$. In fact, the relevant set of stable equilibrium points of \eqref{eq:K1vdpr0} is given as
	\begin{equation}
		\cM_1^\cl=3\left\{ (\br,\be,\bx)\in\R^3\;|\; \be=0,\, \br=3\left(\frac{1}{\bx-\bx^*}-\frac{1}{(\bx-\bx^*)^3}\right),\, \br<\frac{2}{\sqrt{3}} \right\}.
	\end{equation}
	Linearization of \eqref{eq:K1vdpr0} along $\cM_1^\cl$ shows that $\cM_1^\cl$ is exponentially attracting in the $\bx$-direction. Therefore, we can conclude that $\cW_1^\cl$ is located at $\bx=1+\bx^*$, and that it contains the exponentially attracting centre manifold $\cE_1^\cl$ and the curve of exponentially attracting equilibrium points $\cM_1^\cl$.

	\item Note that the flow along the centre manifold $\cW_1^\cl$ has not changed and is given, up to smooth equivalence and away from its corner at $\bx=1+\bx^*$, by
	\begin{equation}
		\begin{split}
			\br' &= \frac{1}{2}\br\\
			\be' &= -\be.
		\end{split}
	\end{equation}
	In other words, the flow along $\cW_1^\cl$ is of saddle type.

\end{enumerate}

With the previous steps we can now prove the main result of this section. For this, let us define the domain $D_1^+=D_1|_{\bx\geq0}$ and the sections
\begin{equation}
	\begin{split}
		\Sigma_1^\txten &= \left\{ (\br,\be,\bx)\in D_1^+\, | \, \be=\delta_1 \right\},\\
		\Sigma_1^\txtex &= \left\{ (\br,\be,\bx)\in D_1^+\, | \, \br=\rho_1, \, 0<\rho_1<\frac{2}{\sqrt{3}} \right\}.
	\end{split}
\end{equation}
	Also, define a small rectangle
\begin{equation}
	R_1=\left\{ (\br,\be,\bx)\in\Sigma_1^\txten\,|\, |\bx-\hcl-\bx^*|\leq\sigma_1, \, \br\leq\tilde\rho_1<\rho_1 \right\},
\end{equation}
\begin{proposition}\label{prop:vdpK1} Consider \eqref{eq:K1vdp} with the controller 
\begin{equation}\label{eq:vdp-u1}
	\bu =-f_1(\br,\be,\bx)-f_1(\br,\be,\bx-\bx^*)+\bv(\be,\bx),
\end{equation}
where 
\begin{equation}
	\begin{split}
		f_1(\br,\be,\bx) &= -1+\bx^2-\frac{1}{2}\bx^2\be-\frac{1}{3}\br\bx^3,\\
		\bv(\be,\bx) &=\frac{2\phi_1+\bx^*}{\phi_1}\left( -1+\phi_1^2-\frac{1}{2}\phi_1^2\be-\frac{1}{3}\br\phi_1^3 \right)-(\be\phi_1+\br\phi_1^2+k_1)(\bx-\bx^*-\phi_1),\\
	\end{split}
\end{equation}
and the function $\phi_1(\br,\be)$ is defined by the open-loop centre manifold $\cW_{1,+}$. Then, one can choose sufficiently small constants $(\delta_1,\rho_1,\sigma_1,\tilde\rho_1)$ such that the following hold for the closed-loop system.

\begin{enumerate}[leftmargin=*]
	\item $D_1^+$ is forward invariant under the flow of \eqref{eq:K1vdp}.
	\item The centre manifold $\cW_1^\cl$ is locally exponentially attracting for $\br\geq0$ sufficiently small, $\be\geq0$ sufficiently small and for $\br^2\be\geq0$ sufficiently small.
	\item If $\bx^*=0$, the centre manifolds $\cW_1^\cl$ and $\cW_{1,+}$ coincide. On the other hand, if $\bx^*<0$ (resp. if $\bx^*>0$) then $\cW_1^\cl$ is located ``to the left'' (resp. ``to the right'') of $\cW_{1,+}$ in the $\bx$-direction.
	\item The image of $R_1$ under the flow of \eqref{eq:K1vdp} is a wedge-like region at $\Sigma_1^\txtex\cap \cM_1^\cl$.
\end{enumerate}
\end{proposition}

\begin{proof}
	The proof follows directly from our previous analysis. In particular, the second item is implied by the stability properties of $\cW_1^\cl|_{\left\{ \br=0 \right\}}$, $\cW_1^\cl|_{\left\{ \be=0 \right\}}$, and the fact that $\br^2\be=\ve$.
\end{proof}

 The closed-loop dynamics corresponding to \eqref{eq:K1vdp} under the controller \eqref{eq:vdp-u1} are as sketched in Figure \ref{fig:K1vdpCL}.

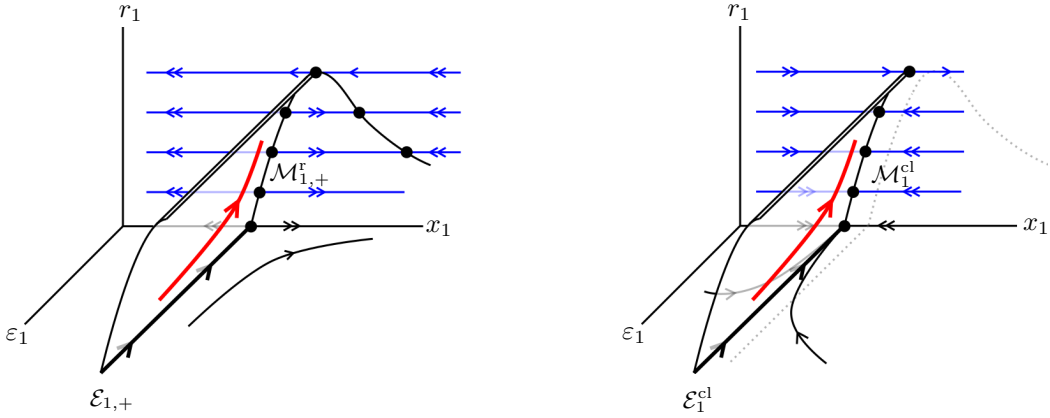
\begin{figure}[htbp]\centering
	
	\begin{tikzpicture}
	\node at (-4,0){
	\begin{tikzpicture}
		\node at (0,0){
		\includegraphics{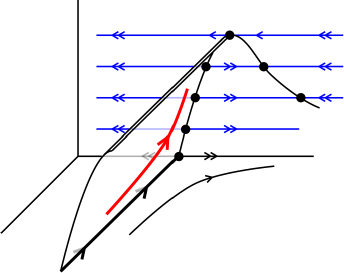}
		};
		\node at (2.6,-.36){$\bx$};
		\node at (-1.5,2.5){$\br$};
		\node at (-3,-1.8){$\be$};
		\node at (-1.75,-2.65){\small$\cE_{1,+}$};
		\node at (.75,.35){\small$\cM_{1,+}^\txtr$};
	\end{tikzpicture}
	
	};	
	\node at (4,0){
	\begin{tikzpicture}
		\node at (0,0){
		\includegraphics{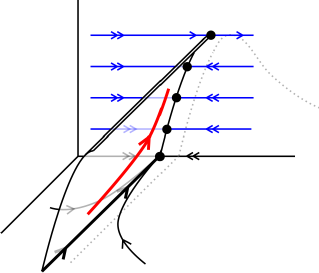}
		};
		\node at (2.55,-.35){$\bx$};
		\node at (-1.4,2.5){$\br$};
		\node at (-2.8,-1.8){$\be$};
		\node at (-1.95,-2.65){\small$\cE_{1}^\cl$};
		\node at (.65,.35){\small$\cM_{1}^\cl$};
	\end{tikzpicture}
	};		
	\end{tikzpicture}	

	\caption{On the left we show the qualitative behavior of the open-loop (that is with $\bu=0$) system \eqref{eq:K1vdp}, while on the right we show the closed-loop system obtained with the controller of Proposition \ref{prop:vdpK1}. In both cases, the $2$-dimensional surface illustrates the centre manifolds $\cW_{1,+}$ on the left and $\cW_{1}^\cl$. The relative position of $\cW_{1}^\cl$ with respect to $\cW_{1,+}$ is determined by $\bx^*$. In the sketch on the right we show that $\cW_1^\cl$ is to the left of $\cW_{1,+}$, which is indicated by the dashed curves.}
	\label{fig:K1vdpCL}
\end{figure}

To finalize this section, we blow-down the controller of Proposition \ref{prop:vdpK1}, as it will be used in the forthcoming section.

\begin{lemma}\label{lemm:u1-vdp} Let $u_1$ denote the blow-down of $\bu$. Then 
	\begin{equation}
		u_1 = -F_0-F_{\bx^*} + v_1,
	\end{equation}
	where
	\begin{equation}
		\begin{split}
			F_{\bx^*}(x,y,\ve) &= -y + (x-\bx^*\sqrt{y})^2 - \frac{(x-\bx^*\sqrt{y})^2\ve}{2y} - \frac{1}{3}(x-\bx^*\sqrt{y})^3\\
			v_1(y,\ve) &= \frac{ 2\phi +\bx^*\sqrt{y} }{\phi}\left( -y + \phi^2-\frac{\ve}{2y}\phi^2-\frac{1}{3}\phi^3 \right) \\
			&- \left( \frac{\ve}{y}\phi+\sqrt{y}\phi^2+k_1\sqrt{y}\right)(x-\phi-\bx^*\sqrt{y}),
		\end{split}
	\end{equation}
	and where $\phi=\phi(y,\ve)$ is defined by $\cS_\ve^\txtr$, that is by $\cS_\ve^\txtr = \left\{ x=\phi(y,\ve) \right\}$.
\end{lemma}
\begin{proof}
	The expression of $u_1$ follows from straightforward computations using \eqref{eq:vdpK1blowup} in \eqref{eq:vdp-u1}. 
	To check $\phi$ is as stated, note that the blow-down induces the relation $\left\{ x_1=\phi_1 \right\} \; \leftrightarrow\; \left\{ x=\sqrt{y}\Phi(\phi_1)=\phi \right\}$, where by $\Phi(\phi_1)$ we are denoting the blow-down of $\phi_1$.

\end{proof}

\subsection{Composite controller and proof of Proposition \ref{thm:canardsvdp}}\label{sec:composite}

In this section we gather the controllers designed in the central chart $K_2$ and in the directional chart $K_1$ into a single one. Our arguments follow from the next general design methodology.

\begin{enumerate}[leftmargin=*]
	\item Let us start with an open-loop vector field $X:\R^N\to\R^N$ such that $X(0)=0$ (here possible parameters $\lambda\in\R^p$ are already included in the vector field by the trivial equation $\dot\lambda=0$).
	\item Let $\cB=\bbS^{N-1}\times \cI$ where $\bbS$ is the unit sphere and $\cI\subseteq\R$ is an interval that contains the origin. Here we shall be interested in $\cI=[0,r_0]$, $r_0>0$. Recall that the blow-up map is defined as $\Phi:\cB\to\R^N$. Moreover, the blow-up transformation induces the so-called ``blown-up'' vector field $\bar X$, which is the vector field that makes the following diagram commute.
	\begin{center}
		\begin{tikzpicture}\centering
		\node (n1) at (0,0) {$\cB$};
		\node (n2) at (2.5,0) {$\R^N$};
		\node (n3) at (0,-2) {$T\cB$};
		\node (n4) at (2.5,-2) {$T\R^N$};
		\draw[->,>=latex] (n1)--(n2) node[midway, above]{$\Phi$};
		\draw[->,>=latex] (n3)--(n4) node[midway, above]{$\txtD\Phi$};
		\draw[->,>=latex] (n1)--(n3) node[midway, left]{$\bar X$};
		\draw[->,>=latex] (n2)--(n4) node[midway, right]{$X$};
	\end{tikzpicture}
	\end{center}
	In other words, $\bar X$ and $X$ are related by the push-forward of $\bar X$ by $\Phi$, that is $\Phi_*(\bar X)=X$, in the sense $\txtD\Phi\circ\bar X=X\circ\Phi$\footnote{Recall that $\bar X$ is well defined: for $r>0$ because of $\Phi|_{\left\{r>0\right\}}$ being a diffeomorphism, and on $r=0$ due to continuous extension to the origin, see \cite{kuehn2015multiple}. Moreover, if the origin is nilpotent, one defines the \emph{desingularized vector field} $\tilde X=\frac{1}{r^k}\bar X$ for some $k>0$, which is smoothly equivalent to $\bar X$ for $r>0$, and all the forthcoming arguments hold equivalently for $\tilde X$.}.
	
	\item Let $\cA=\left\{ (K_i,\Phi_i)\right\}$, with $i=1,\ldots,M$, be a smooth atlas of $\cB$. This means that $(K_i,\Phi_i)$ is a chart of $\cB$, the open sets $K_i$ cover $\cB$, and $\Phi_i:K_i\subset\cB\to\R^N$ is a diffeomorphism. Then, there are local vector fields $\bar X_i$ defined on $K_i$ and given by $\bar X_i=\txtD \Phi_i^{-1}\circ X\circ\Phi_i$.
	
	\item On each chart $K_i$, let us introduce a local controller $\bar u_i$, and define as $\bar X_i^{\cl}\coloneqq\bar X_i + \bar u_i$ the local closed-loop vector field. Naturally, $\bar u_i$ is a local vector field on $K_i$.

	\item Let $\bar\psi_i:K_i\to\R$ be a bump function with compact support $\bar \cN_i\subset K_i$. We choose each $\bar \cN_i$ such that if $K_i\cap K_j\neq\emptyset$ then $\bar \cN_i\cap \bar \cN_j\neq\emptyset$ as well. Note that 
	\begin{equation}
		\bar\varphi_i\coloneqq\frac{\bar\psi_i}{\sum_{i=1}^M\bar\psi_i}
	\end{equation}
	is a partition of unity subordinate to the open cover $\left\{ K_i\right\}_{i=1}^M$.
	\item The sum
	\begin{equation}
		\bar u\coloneqq\sum_{1=1}^M \bar\varphi_i \bar u_i
	\end{equation}
	is, by virtue of the partition of unity, well defined as a global controller on $\cB$. Therefore, the global closed-loop vector field $\bar X^\cl\coloneqq\bar X+\bar u$ is also well defined.

	\item Let us now blow-down $\bar X^\cl$. To be more precise, we now define the closed-loop vector field $X^\cl$ on $\R^N$ by $\Phi_*(\bar X^\cl)=X^\cl$. So, we have
	\begin{equation}
	\begin{split}
		X^\cl = \Phi_*(\bar X^\cl) &= \Phi_*( \bar X+\bar u ) = \Phi_*(\bar X) + \Phi_*(\bar u)=X+ \Phi_*(\bar u),
	\end{split}
	\end{equation}
	where we have used the fact that the push-forward is linear \cite{lee2013smooth}. Next we define $u\coloneqq\Phi_*(\bar u)$ and compute
	\begin{equation}
		\begin{split}
			u=\Phi_*(\bar u) &=\Phi_*\left( \sum_{i=1}^M\bar\varphi_i\bar u_i \right)=\sum_{i=1}^M\Phi_*(\bar\varphi_i\bar u_i)=\sum_{i=1}^M(\Phi_i)_*(\bar\varphi_i\bar u_i)\\
			&=\sum_{i=1}^M \varphi_i\cdot(\Phi_i)_*(\bar u_i),
		\end{split}
	\end{equation}
	where $\varphi_i\coloneqq\bar\varphi_i\circ\Phi_i^{-1}$ for $i=1,\ldots,M$, and it is clear from its definition that $\left\{ \varphi_i \right\}$ is a partition of unity a neighborhood of the origin $0\in\R^N$ subordinate to the open cover $\left\{ \Phi_i(K_i) \right\}$. 

\end{enumerate}

With the previous methodology we define the controller that stabilizes canard cycles of the van der Pol oscillator as 
\begin{equation}\label{eq:comp-u}
	u=\frac{1}{2}\psi_1u_1+\frac{1}{2}\psi_2u_2,
\end{equation}
where $u_1$ is as given by Lemma \ref{lemm:u1-vdp} and $u_2$ as in Theorem \ref{thm:fold-fast}, and where $\psi_1$ is a bump function with support $\cN_1$ containing the repelling branch $\cS_0^\txtr$ and $\cN_2$ the parabola $\left\{ y=x^2 \right\}$ around the origin. \ch{Although several choices for these neighbourhoods are possible, we recall an example given at the beginning of section \ref{sec:vdp}: } 
\begin{equation}
	\begin{split}
		\cN_1 &= \left\{ (x,y)\in\R^2 \, : \, |-y+x^2-\frac{1}{3}x^3|<\beta_1,\, 0<x<2,\, y_{\txtmin}<y<y_{h} \right\}\\
		\cN_2 &= \left\{ (x,y)\in\R^2 \, : \, |-y+x^2|<\beta_2, \, -x_\txtmin<x<x_{\txtmax} \right\},
	\end{split}
\end{equation}
with suitably chosen positive constants $\beta_1$, $\beta_2$, $x_{\txtmin}$, $x_{\txtmax}$, $y_{\txtmin}$, $y_{h}<\frac{4}{3}$. We note that one must choose \ch{$0<y_\txtmin\in\cO(\ve)$} in order to ensure that the slow manifold $\cS_\ve^\txtr$ is within distance $\cO(\ve)$ of the critical manifold $\cS_0^\txtr$. Here $y_h$ controls the height of the desired canard cycle, therefore $y_h<\frac{4}{3}$. The neighborhood $\cN_1$ and $\cN_2$ are sketched in Figure \ref{fig:vdp1}.

With the controller as in \eqref{eq:comp-u}, and given the analysis in Section \ref{sec:K1vdp}, one has that orbits of \eqref{eq:vdp} passing close to the origin follow closely the repelling branch of the slow manifold $\cS_\ve^\txtr$ up to a height determined by $y_h$. Once orbits leave the neighborhood $\cN_1\cup\cN_2$, they are governed by the open-loop dynamics. Finally the controller of Proposition \ref{thm:canardsvdp} is indeed \eqref{eq:comp-u}. We have just dropped the subscript of the constant $x_1^*$.


\section{Conclusions and Outlook}\label{sec:conclusions}

In this paper we have presented a methodology combining the blow-up method with Lyapunov-based control techniques to design a controller that stabilizes canard cycles. The main idea is to use a first integral in the blow-up space to regulate the canard cycle that the orbits are to follow. Later on, we have extended the previously developed method to control canard cycles in the van der Pol oscillator. Roughly speaking this procedure follows two steps: first one needs a controller that stabilizes a folded maximal canard within a small neighborhood of the canard point. Next, one needs to stabilize the unstable branch of the open-loop slow manifold and to tune the position of the closed-loop orbits with respect to it. This is essential to determine whether the closed-loop canard has a head or not. Finally one combines such controllers by means of a partition of unity. We have further shown that the proposed controller can be used to produce stable MMOs.

Several new questions and possible extensions arise from our work, and we would like to finish this paper by briefly mentioning a couple of ideas. First of all, it becomes interesting to adapt the controllers designed here to neuron models such as the FitzHugh-Nagumo, Morris-Lecar, or Hodgkin-Huxley models. Another relevant extension is to develop optimal controllers to control canards. Although from a theoretical point of view one would be interested in arbitrary cost functionals, some particular choices might be more suitable for applications. For instance, one may one want to design minimal energy controllers. It is also not completely clear whether the strategy of combining the blow-up method and control techniques still applies as the optimal controllers may be time-dependent. Finally, the notion of controlling MMOs definitely requires further investigation, as here we have just given a simple sample of the possibilities. Thus, for example, extending the ideas of this paper to higher-dimensional fast-slow systems with non-hyperbolic points is a direction to be pursued in the future.

\textbf{Acknowledgments:} The authors thank the anonymous reviewer for the comments that helped to improve this manuscript. HJK gratefully acknowledges support by a fellowship of the Alexander-von-Humboldt Foundation. CK acknowledges partial support by the DFG via the SFB/TR109 Discretization in Geometry and Dynamics and support by a Lichtenberg Professorship of the VolkswagenFoundation.

\bibliographystyle{plain} 
\bibliography{bib}
\end{document}